\documentclass{amsart}

\usepackage[backend=biber,maxbibnames=5,maxalphanames=5,style=alphabetic,bibencoding=utf8,giveninits,url=false,isbn=false]{biblatex}
\AtBeginBibliography{\small}

\usepackage{verse}

\newlength\lineindent

\usepackage[all,hyperref]{bi-discrete}
\usepackage{cleveref}
\usepackage{mathtools}
\usepackage{algorithm2e}

\providecommand{\dx}{\, \mathrm{d}x}

\providecommand{\tria}{\mathcal{T}}
\providecommand{\cD}{\mathcal{D}}

\providecommand{\supp}{\textup{supp}}

\newcommand{\js}[1]{#1}

\usepackage{dsfont}
\usepackage{pgfplots}
\usepackage{tikz}
\usepackage{listings}
\usepackage{stmaryrd}

\usepackage{tikz}
\usepackage{pgfplots}
\usepackage{pgfplotstable} % Required for table manipulation
\pgfplotsset{
  width=.65\linewidth,
  axis background/.style={fill=black!5!white},
  grid style={densely dotted,semithick},
  legend style={
    legend columns=1,
    legend pos=outer north east
  },
  compat=newest % compatibility for old pgfplots versions
}
\pgfplotscreateplotcyclelist{MyColors2}{%
    {red,mark = *,every mark/.append style={solid,scale=0.6,fill=red}},
	{teal,mark = x,every mark/.append style={solid,scale=0.6,fill=teal}},
    {green,mark = diamond*,every mark/.append style={solid,scale=0.6,fill=green}},    
    {blue,mark = square*,every mark/.append style={solid,scale=0.6,fill=blue}},    
    {blue,mark = square*,every mark/.append style={solid,scale=0.6,fill=blue}},    
    {blue,mark = square*,every mark/.append style={solid,scale=0.6,fill=blue}},    
    }

\pgfplotscreateplotcyclelist{MyColors}{%
    {red,mark = *,every mark/.append style={solid,scale=0.6,fill=red}},
	{teal,mark = x,every mark/.append style={solid,scale=0.6,fill=teal}},
    {blue,mark = square*,every mark/.append style={solid,scale=0.6,fill=blue}},    
    {dotted,red,mark = *,every mark/.append style={solid,scale=0.6,fill=red}},
	{dotted,teal,mark = x,every mark/.append style={solid,scale=0.6,fill=teal}},
    {dotted,blue,mark = square*,every mark/.append style={solid,scale=0.6,fill=blue}},
    }

\makeatletter
\@namedef{subjclassname@2020}{%
  \textup{2020} Mathematics Subject Classification}
\makeatother

\begin{document}
\lstset{language=Python,
basicstyle=\small, % alle listings winzig drucken (meine Standardeinstellung)
keywordstyle=\color{black}\bfseries, % Schlüsselwörter fett und schwarz drucken
commentstyle=\color{blue}, % Kommentare blau drucken
stringstyle=\ttfamily, % Strings im Code Schreibmaschinenähnlich - setzt sich etwas vom Code ab
showstringspaces=false,
numbers=left, % Nummerierung der Linien links vom Code
numberstyle=\small, % Art der Nummerierung, hier z. B.: winzige Schrift
%stepnumber=5, % Intervall der Nummern, hier z.B. 5
              % falls nicht gesetzt: Alle Nummern werden ausgegeben.
numbersep=10 pt,
xleftmargin= 27pt,
%xrightmargin= 10pt
}

%\author[L.~Diening]{Lars Diening}
\author[J.~Storn]{Johannes Storn}
\address[J.~Storn]{Faculty of Mathematics \& Computer Science, Institute of Mathematics, Leipzig University, Augustusplatz 10, 04109 Leipzig, Germany}
\email{johannes.storn@uni-leipzig.de}


\thanks{The work of the author was supported by the Deutsche Forschungsgemeinschaft (DFG, German Research Foundation) -- SFB 1283/2 2021 -- 317210226.}

\title[Randomized Projection Operators onto Piecewise Polynomial Spaces]{Randomized Projection Operators onto Piecewise Polynomial Spaces}

\begin{abstract}
We introduce randomized projection operators onto piecewise polynomial spaces, defined via sampling and discrete least-squares polynomial approximations. The resulting mappings are computable for any function in $L^2$ and exhibit (almost) optimal approximation properties in $L^2$ and $H^{-1}$. As smoothers for incomplete or rough data, they yield computable finite element discretizations with optimal rates of  convergence.
\end{abstract}
\maketitle

\section{Introduction}
In recent years, the numerical approximation of elliptic PDEs with rough right-hand sides has attracted significant attention. Starting with the pioneering contributions of Veeser and Zanotti \cite{VeeserZanotti18,VeeserZanotti18b,VeeserZanotti19}, there is by now a substantial literature on smoothing and projection operators that enable quasi-optimal nonconforming, mixed, and least-squares finite element schemes in the presence of data of limited regularity; see for example \cite{FuehrerHeuerKarkulik22,CarstensenNataraj22,DieningStornTscherpel21b,Fuehrer24,CarstensenGraessleNataraj24}. Closely related questions arise in a posteriori error control, where the treatment of data-oscillation and the approximation of load functionals plays a central role \cite{KreuzerVeeser21}.
A natural but rarely addressed issue concerns the evaluation of such rough right-hand sides. Indeed, classical deterministic quadrature rules for the computation of load vectors typically require additional smoothness of the integrand such as piecewise $W^{1,\infty}$ regularity used in \cite[Thm.~33.17]{ErnGuermond21b}, see \cite{Fix72,Ciarlet2002,BabuskaBanerjeeLi11} for further results. This effect negates the benefits of smoothing operators and the resulting quasi-optimal schemes when the right-hand side $f\in L^2(\cD)$ is truly irregular. To remedy this drawback, we modify an idea of \cite{KrusePolydoridesWu19} by exploiting  randomization. 
Instead of applying the randomized quadrature directly to the load $f$, we design randomized smoothing operators $\hat{\Pi}$ that are projections onto piecewise polynomials, allowing the exact evaluation of the resulting approximated load $\hat{\Pi} f \approx f$. 
\js{For the lowest-order case, our design bases on (stratified) Monte Carlo quadrature. For higher-order cases, we combine piece-wise unweighted randomized least-squares approximations with the lowest-order operator.}
This ansatz offers two major advantages:
\begin{enumerate}
\item The independence and unbiasedness of the lowest-order randomized projection $\hat{\Pi}_0$ enables a suitable diagonalization of the $H^{-1}(\cD)$ norm (see Lemma~\ref{lem:Hm1}), leading to an expected error in $H^{-1}(\cD)$ bounded from above by the data-oscillation without any additional requirements on the smoothness beyond $f\in L^2(\cD)$.\label{itm:AA}
\item \js{For the higher-order randomized piecewise polynomial projection we obtain an expected $L^2(\cD)$ approximation error that is bounded from above by the best-approximation errors with respect to $\lVert \bigcdot \rVert_{L^p(\cD)}$ for $p>2$. We derive an explicit bound on the required number of samples. A modification of the operator allows us to control the expected $H^{-1}(\cD)$ error by the weighted norm $\lVert h_\tria \bigcdot \rVert_{L^p(\cD)}$ with local mesh size $h_\tria$.}\label{itm:BB}
\end{enumerate}
The lowest-order operator $\hat{\Pi}_0$ in \ref{itm:AA} is introduced and discussed in Section~\ref{sec:LowOrder}. The higher-order operators in \ref{itm:BB} are investigated in Section~\ref{sec:HigherOrder}. We illustrate their application as a smoother for rough right-hand sides in Section~\ref{sec:Smoother}. Numerical experiments in Section~\ref{sec:Exp} underline the theoretical findings and compare the suggested scheme to deterministic approaches. 

Throughout the paper, all random variables are defined on a probability space $(\Omega,\mathcal{F},\mathbb{P})$, and expectations $\mathbb{E}$ are taken with respect to the measure $\mathbb{P}$.

\js{\subsection*{Relation to existing literature}
While the study of randomized quadrature for assembling load vectors in finite element computations is, to the author's knowledge, limited to \cite{KrusePolydoridesWu19}, there exists a vast and rapidly growing literature on randomized quadrature rules in general. Let us hence compare our approximation results to existing findings in the literature.

In the lowest-order case, our $H^{-1}(\cD)$ analysis can be viewed as an application of Monte Carlo approximation of Hilbert-space-valued Bochner integrals as in \cite{BarthLang12,BarthLangSchwab13}, see Remark~\ref{rem:AbstrFrameworkLowOrder} for details. The essential difference to \cite{KrusePolydoridesWu19} is that the randomized quadrature therein is applied directly to the load functional. On each cell $K$ of the underlying triangulation $\tria$, it takes the form
\begin{align*}
\frac{|K|}{N_K}\sum_{i=1}^{N_K}f(X_i^K)v_h(X_i^K).
\end{align*}
Hence, every sampled value is represented by the point load $|K|f(X_i^K)\delta_{X_i^K}$. For $d\geq2$, such point loads do not belong to $H^{-1}(\cD)$ and are naturally interpreted only as functionals on the finite-dimensional test space. Correspondingly, the energy-error analysis in \cite{KrusePolydoridesWu19} relies on discrete estimates and, for direct quadrature of the load, requires additional integrability $f\in L^p(\cD)$ with $p>2$ to obtain a positive convergence rate.
By contrast, our construction replaces the singular measure $|K|\delta_{X_i^K}$ by the spatially distributed density $\indicator_K$.
The resulting randomized approximation is therefore an $H^{-1}(\cD)$-valued random variable. This regularization permits the Hilbert-space-valued Monte Carlo argument and yields in combination with Lemma~\ref{lem:indicator} the bound by the data-oscillation in Lemma~\ref{lem:Hm1}. 

Our higher-order projection operator can be seen as a two-stage randomized method. The first stage uses (a stratified version of) unweighted randomized least-squares approximations.
From the viewpoint of information-based complexity, this is a cellwise instance of approximation
from i.i.d.\ standard information; see \cite[Sec.~3.1 and 4.1]{SonnleitnerUllrich23}. Our error criterion is, however, of randomized-algorithm type, since we control the expected squared error for each input function rather than the deterministic worst-case error associated with a fixed realization of the random information; cf.~\cite[Sec.~5.3]{SonnleitnerUllrich23}.
Classical analyses of such estimators typically control the empirical Gram matrix only on an event of high probability. Expected error estimates then require truncation or contain an additional term arising from the complementary event, see for example \cite{CohenDavenportLeviatan13,MiglioratiNobileSchwerinTempone14,CohenGiovanni17,Adcock25}. We instead control the lower tail of the smallest eigenvalue down to arbitrarily small values and combine this estimate with H\"older's inequality. In this way, the contribution of nearly singular realizations is absorbed into the stronger $L^p(K)$-norm of the best-approximation error, with $p>2$, and no additional failure-probability term occurs.
This approach is closely related to the abstract analysis in \cite{Mourtada22}. Its application to the present setting is, however, not immediate. For the polynomial feature vector $\psi_K(X)\in\mathbb R^m$, we have to verify a quantitative small-ball condition uniformly over all polynomial directions and suitable moment bounds. 
These polynomial-specific arguments also yield sharper explicit conditions on the number of samples than a direct application of \cite{Mourtada22}. Note that the first stage could use alternative randomized operators such as weighted randomized least-squares approximations, which improve upon our result in the sense that one obtains best-approximation with respect to $p=2$ but require more evolved sampling procedures, see for example \cite{CohenMigliorati17,AdcockCardenas20,Migliorati21,DolbeaultCohen22} and for a general overview \cite[Secs.~3.1 and 5.3]{SonnleitnerUllrich23} and \cite[Sec.~7.3]{KriegUllrich26}. The second stage of our two-stage randomized method is a correction, motivated by the ``Fortin trick'': We use our lowest-order operator to correct the randomized least-squares approximation. 
This shares similarities with the use of control-variates, see for example \cite[Sec.~9.3]{KroeseTaimreBotev11}.
However, in our setting it is not merely a tool to improve the variance, but to restore the
cellwise mean condition which extends the $L^2(\cD)$ approximation properties to $H^{-1}(\cD)$.
}

\section{Lowest-order projection}\label{sec:LowOrder}
In this section we discuss a randomized projection $\hat{\Pi}_0 \colon L^2(\cD) \to \mathbb{P}_0(\js{\Omega;}\tria)$ onto piecewise constants via a standard Monte Carlo quadrature. 
In particular, let $f\in L^2(\cD)$ and let $(N_K)_{K\in \tria} \subset \mathbb{N}$. Given independent and uniformly distributed random variables $X_1^K,\dots,X_{N_K}^K \sim \mathcal{U}(K)$ for all $K\in \tria$, we define the piecewise constant Monte Carlo approximation
\begin{align}\label{eq:LowOrderRI}
\hat{f} \coloneqq \hat{\Pi}_0 f  \coloneqq \sum_{K\in \tria}\hat f_K \indicator_K\qquad\text{with }\hat f_K \coloneqq \frac{1}{N_K} \sum_{i=1}^{N_K} f(X_i^K).
\end{align}
The estimator $\hat{f}$ is unbiased and enjoys beneficial local approximation properties in $L^2(\cD)$, as shown in the following two lemmas. They involve the $L^2(\cD)$ orthogonal projection $\Pi_0 \colon L^2(\cD) \to \mathbb{P}_0(\tria)$ onto piecewise constants, equivalently characterized by the integral mean property
\begin{align}\label{eq:IntMeanProp}
(\Pi_0 f)|_K = f_K \coloneqq \frac{1}{|K|} \int_K f \dx \qquad\text{for all }K\in \tria.
\end{align}
\begin{lemma}[Unbiased approximation of $\Pi_0$]\label{lem:UnbiasedPi0}
For any $f\in L^2(\cD)$ one has the identity 
\begin{align*}
\mathbb{E}\big[\hat \Pi_0 f\big] = \Pi_0 f.
\end{align*}
\end{lemma}
\begin{proof}
Let $K\in \tria$ and $f\in L^2(\cD)$. Since the random variables $X_1^K,\dots,X_{N_K}^K \sim \mathcal{U}(K)$ are uniformly distributed, one has 
\begin{align*}
\mathbb{E}\big[\hat{f}_K\big] = \frac{1}{N_K} \sum_{i=1}^{N_K} \mathbb{E}\big[ f(X_i^K)\big] =  \frac{1}{N_K}  \sum_{i=1}^{N_K} \frac{1}{|K|} \int_K f(x) \dx = \frac{1}{|K|} \int_K f\dx. 
\end{align*}
Combining this observation with \eqref{eq:IntMeanProp} concludes the proof.
\end{proof}
\begin{lemma}[Approximation properties in $L^2(K)$]\label{lem:ApxPi0L2}
One has for any $f\in L^2(\cD)$ and  $K\in \tria$ the identity
\begin{align*}
\mathbb{E}\big[\lVert f - \hat \Pi_0 f \rVert^2_{L^2(K)} \big] = \lVert f- \Pi_0 f \rVert_{L^2(K)}^2 + \mathbb{E}\big[\lVert \Pi_0 f - \hat{\Pi}_0 f \rVert^2_{L^2(K)} \big]. 
\end{align*}
The latter term equals
\begin{align*}
\mathbb{E}\big[\lVert \Pi_0 f - \hat{\Pi}_0 f \rVert^2_{L^2(K)} \big] = \frac{1}{N_K} \lVert f - \Pi_0 f \rVert_{L^2(K)}^2.
\end{align*}
\end{lemma}
\begin{proof}
Let $K\in \tria$.
The first identity in the lemma follows by the orthogonality of $\Pi_0$ and the Pythagorean theorem. Moreover, the independence of the random variables $X_1^K,\dots,X_{N_K}^K$ and the property $\mathbb{E}\big[f_K -  f(X_i^K)\big] = 0$ for all $i=1,\dots,N_K$, which follows by Lemma~\ref{lem:UnbiasedPi0}, yield
\begin{align*}
\mathbb{E}\big[ |f_K - \hat{f}_K|^2 \big] &= \frac{1}{N^2_K} \mathbb{E}\Big[  \Big|\sum_{i=1}^{N_K}\big(f_K -  f(X_i^K)\big)\Big|^2  \Big] = \frac{1}{N^2_K} \sum_{i=1}^{N_K} \mathbb{E}\Big[  \big|f_K -  f(X_i^K)\big|^2  \Big]\\
& = \frac{1}{N_K}\frac{1}{|K|} \int_K |f_K - f|^2\dx.
\end{align*}
This concludes the proof of the lemma's second identity. 
\end{proof}
Apart from approximation properties in $L^2(\cD)$, applying our projection operator to right-hand sides of PDEs as discussed in Section~\ref{sec:Smoother} below or to time-derivatives in parabolic problems motivates error estimates with respect to the dual norm 
\begin{align*}
\lVert g \rVert_{H^{-1}(\cD)} \coloneqq \sup_{v\in H^1_0(\cD)\setminus \lbrace 0 \rbrace} \frac{\int_\cD gv\dx}{\lVert \nabla v \rVert_{L^2(\cD)}}\qquad\text{for all }g\in L^2(\cD).
\end{align*}
For the corresponding error estimate, we exploit the following auxiliary result.
\begin{lemma}[$H^{-1}(\cD)$ norm of locally supported functions]\label{lem:indicator}
Let $g\in L^2(\cD)$ with $\supp(g) \subset K \in \tria$.
Then there exists a constant $C>0$ depending only on the domain $\cD$ and the shape-regularity of $K$ such that with 
\begin{align}\label{eq:DefVartheta}
\vartheta(h_K) \coloneqq \begin{cases}
C\, \max\lbrace 1, \ln(h_K^{-1})\rbrace^{1/2}&\text{for }d=2,\\
C &\text{for }d \geq 3
\end{cases}
\end{align}
we have
\begin{align*}
\lVert g\rVert_{H^{-1}(\cD)} \leq h_K \vartheta(h_K) \lVert g \rVert_{L^2(K)}.
\end{align*}
\end{lemma}
\begin{proof}
Let $g\in L^2(\cD)$ with $\supp(g) \subset K\in \tria$. 
The Cauchy--Schwarz inequality implies
\begin{align}
\lVert g \rVert_{H^{-1}(\cD)} \coloneqq \sup_{v\in H_0^1(\cD)\setminus \lbrace 0 \rbrace}
\frac{\int_K g v\dx}{\lVert \nabla v\rVert_{L^2(\cD)}}\leq \lVert g \rVert_{L^2(K)}\,\sup_{v\in H_0^1(\cD)\setminus \lbrace 0 \rbrace}
\frac{ \lVert v \rVert_{L^2(K)}}{\lVert \nabla v\rVert_{L^2(\cD)}}.
\label{eq:indicator_dual_def}
\end{align}
H\"older's inequality yields for any $v\in H^1_0(\cD)$ and exponents $q\in[2,\infty)$
\begin{align*}
\lVert v \rVert_{L^2(K)}^2 \leq \lVert v^2 \rVert_{L^{q/2}(K)} \lVert 1 \rVert_{L^{q/(q-2)}(K)} = \lVert v \rVert_{L^{q}(K)}^2 |K|^\frac{q-2}{q}
\end{align*}
and consequently
\begin{align*}
\lVert v \rVert_{L^2(K)} \leq |K|^{\frac{1}{2} -\frac{1}{q}} \lVert v \rVert_{L^{q}(K)} \leq |K|^{\frac{1}{2} -\frac{1}{q}} \lVert v \rVert_{L^{q}(\cD)}.
\end{align*}
Inserting this bound into \eqref{eq:indicator_dual_def} leads to
\begin{align}
\lVert g\rVert_{H^{-1}(\cD)}
\leq |K|^{\frac12 - \frac{1}{q}} \lVert g \rVert_{L^2(K)} \sup_{v\in H_0^1(\cD)\setminus \lbrace 0 \rbrace}\frac{\lVert v\rVert_{L^q(\cD)}}{\lVert \nabla v\rVert_{L^2(\cD)}}.
\label{eq:indicator_holder_q}
\end{align}

 \textit{Case $d\ge 3$.} We set the exponent $q\coloneqq 2d/(d-2)$. The Sobolev embedding
$H_0^1(\cD)\hookrightarrow L^{q}(\cD)$, see e.g.~\cite[Thm.~2.31]{ErnGuermond21}, verifies the existence of a constant $C< \infty$ such that
\begin{align*}
\lVert v\rVert_{L^{q}(\cD)} \leq C\, \lVert\nabla v\rVert_{L^2(\cD)}\qquad\text{for all }v\in H^1_0(\cD).
\end{align*}
Combining this bound with \eqref{eq:indicator_holder_q} gives
\begin{align*}
\lVert g\rVert_{H^{-1}(\cD)} \leq C\,|K|^{\frac12-\frac{d-2}{2d}} \lVert g \rVert_{L^2(K)} = C\,|K|^{\frac1d}\lVert g \rVert_{L^2(K)}.
\end{align*}
The shape regularity of $K$ implies $|K|^{1/d}\simeq h_K$. This yields the existence of a constant $C<\infty$ with 
\begin{align*}
\lVert g \rVert_{H^{-1}(\cD)}\leq C\,h_K \lVert g \rVert_{L^2(K)}.
\end{align*}

\textit{Case $d=2$.}
In two dimensions the embedding constant for $H_0^1(\cD)\hookrightarrow L^q(\cD)$ grows with $\sqrt{q}$ for all $q\in [2,\infty)$; that is, with constant $C< \infty$ independent of $q$ one has \cite[Lem.~2.3]{KozonoOgawaSohr92}
\begin{align*}
\lVert v\rVert_{L^q(\cD)} \leq C\sqrt{q}\,\lVert \nabla v\rVert_{L^2(\cD)} \qquad\text{for all } v\in H_0^1(\cD)\text{ and }q\in[2,\infty).
\end{align*}
Combining this bound with \eqref{eq:indicator_holder_q} leads to
\begin{align*}
\lVert g \rVert_{H^{-1}(\cD)} \leq C\sqrt{q}\,|K|^{\frac12-\frac1q} \lVert g \rVert_{L^2(K)} .
\end{align*}
Let $q\coloneqq \max\lbrace 2, \ln(|K|^{-1})\rbrace$. If $|K|\geq e^{-2}$, then $q=2$ and hence $|K|^{-1/q}=|K|^{-1/2}\leq e$. If $|K|< e^{-2}$, then $q=\ln(|K|^{-1}) = - \ln(|K|)$ and therefore
\begin{align*}
|K|^{-\frac{1}{q}} = |K|^{\frac{1}{\ln(|K|)}} = e.
\end{align*}
In both cases, $|K|^{-1/q}\leq e$ and we obtain
\begin{align*}
\lVert g \rVert_{H^{-1}(\cD)} \leq Ce\,|K|^\frac{1}{2} \,\sqrt{\max\lbrace 2, \ln(|K|^{-1}) \rbrace }\, \lVert g \rVert_{L^2(K)}.
\end{align*}
The shape regularity of $K$ implies $|K|\simeq h_K^2$ and thus $\ln(|K|^{-1})\simeq \ln(h^{-2}_K) = 2\,\ln(h^{-1}_K)$. This yields the lemma for $d= 2$.
\end{proof}
We can now verify the following estimate.
\begin{lemma}[Expected error for cell averages]\label{lem:Hm1}
Let $f\in L^2(\cD)$.
Its piecewise constant randomized approximation $\hat \Pi_0 f$ satisfies 
\js{\begin{align*}
\mathbb E \big[ \lVert f - \hat \Pi_0 f\rVert_{H^{-1}(\cD)}^2\big] = \lVert f - \Pi_0 f\rVert_{H^{-1}(\cD)}^2 + \mathbb E \big[ \lVert \hat \Pi_0 f - \Pi_0 f\rVert_{H^{-1}(\cD)}^2\big].
\end{align*}}
The deterministic term is bounded by 
\begin{align*}
\lVert f - \Pi_0 f\rVert_{H^{-1}(\cD)}^2 \leq \pi^{-2} \sum_{K\in \tria} h^2_K\lVert f - \Pi_0 f\rVert_{L^2(K)}^2.
\end{align*}
The stochastic term satisfies
\begin{align*}
\mathbb E \big[ \lVert  \hat \Pi_0 f - \Pi_0 f\rVert_{H^{-1}(\cD)}^2\big] \leq \sum_{K\in\tria} h_K^2 \frac{\vartheta(h_K)^2}{N_K} \lVert  f-\Pi_0 f \rVert^2_{L^2(K)}.
\end{align*}
\end{lemma}
\begin{proof}
\textit{Step 1 (Bound for the deterministic term).}
Using \Poincare's inequality on convex domains and the orthogonality of $\Pi_0$, we bound the first addend by
\begin{align*}
\lVert f - \Pi_0 f\rVert_{H^{-1}(\cD)}  &= \sup_{v\in H^1_0(\cD)\setminus \lbrace 0 \rbrace} \frac{\int_\cD (f - \Pi_0 f) (v-\Pi_0 v)\dx }{\lVert \nabla v \rVert_{L^2(\cD)}} \\
&  \leq \sup_{v\in H^1_0(\cD)\setminus \lbrace 0 \rbrace} \frac{\lVert f - \Pi_0 f\rVert_{L^2(\cD)}\lVert v - \Pi_0 v\rVert_{L^2(\cD)}}{\lVert \nabla v \rVert_{L^2(\cD)}}\\
& \leq \pi^{-1} \Big(\sum_{K\in \tria} h^2_K\lVert f - \Pi_0 f\rVert_{L^2(K)}^2\Big)^{1/2}.
\end{align*}

\textit{Step 2 (Bound for the stochastic term).}
To bound the stochastic term, we write the difference as a sum of independent cell contributions $\hat g_K \coloneqq \hat f_K-f_K \in \mathbb{R}$ for all $K\in \tria$ in the sense that 
\begin{align*}
\hat{\Pi}_0 f - \Pi_0 f = \sum_{K\in \tria} \hat g_K \indicator_K.
\end{align*}
Lemma~\ref{lem:UnbiasedPi0} shows $\mathbb E[\hat g_K]=0$ for all $K\in \tria$.
Let  $A \coloneqq -\Delta\colon H_0^1(\cD)\to H^{-1}(\cD)$ denote the negative Laplacian with inverse $A^{-1}$. 
One has for any $\xi \in H^{-1}(\cD)$ the identity 
\begin{align*} 
\lVert \xi \rVert_{H^{-1}(\cD)} &= \sup_{v\in H^1_0(\cD)\setminus\lbrace 0 \rbrace} \frac{ \langle \xi, v\rangle_\cD }{\lVert \nabla v\rVert_{L^2(\cD)}} = \sup_{v\in H^1_0(\cD)\setminus\lbrace 0 \rbrace} \frac{\int_\cD \nabla A^{-1} \xi \cdot \nabla v\dx }{\lVert \nabla v\rVert_{L^2(\cD)}}\\
& = \lVert \nabla A^{-1}\xi \rVert_{L^2(\cD)} = \langle \xi,A^{-1}\xi\rangle_\cD^{1/2}.
\end{align*}
Hence, we have
\begin{align}\label{eq:DiagTerms}
\begin{aligned}
\mathbb E\big[\lVert \hat \Pi_0 f- \Pi_0 f\rVert_{H^{-1}(\cD)}^2\big] &= \mathbb E\Big[\Big\langle \sum_{K\in \tria} \hat g_K\indicator_K, A^{-1}\sum_{L\in \tria}  \hat g_L\indicator_L\Big\rangle_\cD \Big] \\
&= \sum_{K,L \in \tria}\mathbb E[\hat g_K\hat g_L]\, \langle \indicator_K, A^{-1}\indicator_L\rangle_\cD.
\end{aligned}
\end{align}
The independence of the random variables $\hat g_K$ and their expected value of zero imply $\mathbb E[\hat g_K\hat g_L]=0$ for simplices $K,L\in \tria$ with $K\neq L$. Consequently, we have
\begin{align}\label{eq:Pasdafgwqrw}
\begin{aligned}
\mathbb E\big[ \lVert \hat \Pi_0 f - \Pi_0 f\rVert_{H^{-1}(\cD)}^2\big] &= \sum_{K\in\tria}\mathbb E[\hat g_K^2]\, \langle \indicator_K, A^{-1}\indicator_K\rangle_\cD\\
&= \sum_{K\in\tria} \frac{1}{|K| }\mathbb{E}\big[\lVert \hat{\Pi}_0 f - \Pi_0 f \rVert_{L^2(K)}^2\big]  \lVert \indicator_K\rVert_{H^{-1}(\cD)}^2.
\end{aligned}
\end{align}
Applying Lemma \ref{lem:ApxPi0L2} and Lemma~\ref{lem:indicator} with $g \coloneqq \indicator_K$  leads to the bound
\begin{align*}
\mathbb E\big[ \lVert \hat \Pi_0 f- \Pi_0 f\rVert_{H^{-1}(\cD)}^2\big] & \js{ =} \sum_{K\in\tria}\frac{\lVert \indicator_K\rVert_{H^{-1}(\cD)}^2}{N_K|K|}\lVert  f - \Pi_0 f\rVert^2_{L^2(K)}\\
& \leq  \sum_{K\in\tria} h_K^2 \frac{\vartheta(h_K)^2}{N_K} \lVert  f-\Pi_0 f \rVert^2_{L^2(K)}.
\end{align*}

\js{\textit{Step 3 (Orthogonality in expectation).}
Splitting the error results in
\begin{align*}
\mathbb E\big[\lVert f-\hat\Pi_0f\rVert_{H^{-1}(\cD)}^2\big] & = \lVert f-\Pi_0f\rVert_{H^{-1}(\cD)}^2 + \mathbb E\big[\lVert \Pi_0f-\hat\Pi_0f\rVert_{H^{-1}(\cD)}^2\big] \\
&\quad + 2\mathbb E\big[ \langle f-\Pi_0f,A^{-1}(\Pi_0f-\hat\Pi_0f)\rangle_\cD \big].
\end{align*}
Since $f-\Pi_0f$ is deterministic, Lemma~\ref{lem:UnbiasedPi0} yields
\begin{align*}
&\mathbb E\big[ \langle f-\Pi_0f,A^{-1}(\Pi_0f-\hat\Pi_0f)\rangle_\cD \big] = \langle f-\Pi_0f,A^{-1}\mathbb E[\Pi_0f-\hat\Pi_0f] \rangle_\cD = 0.
\end{align*}
Combining the identities concludes the proof.}
\end{proof}
\begin{remark}[Diagonalization]
Decompositions similar to the one in \eqref{eq:DiagTerms} result for naive deterministic interpolation operators $\mathcal{I}\colon L^2(\cD)\to \mathbb{P}_k(\tria)$ in off-diagonal entries that do not vanish. Hence, one is often restricted to using \Poincare's inequality globally; that is,
\begin{align*}
\lVert f- \mathcal{I} f \rVert_{H^{-1}(\cD)}\lesssim \textup{diam}(\cD) \lVert f- \mathcal{I}f \rVert_{L^2(\cD)}\qquad\text{for }f\in L^2(\cD).
\end{align*}
This ``loses'' a power of $h_K$. Remedying this drawback requires sophisticated designs, see for example~\cite{TantardiniVeeser16,DieningStornTscherpel21b,Fuehrer24}.
\end{remark}
\js{\begin{remark}[Abstract framework]\label{rem:AbstrFrameworkLowOrder}
The orthogonality in expectation and the identity in \eqref{eq:Pasdafgwqrw} can be interpreted in the abstract framework of Monte Carlo approximation of Hilbert-space-valued Bochner integrals; see for example \cite{BarthLang12,BarthLangSchwab13}.
Indeed, equip any $K \in \tria$ with the probability measure $\mu_K\coloneqq |K|^{-1}\dx$ and define the  mapping $\Phi_K\in L^2(K,\mu_K;H^{-1}(\cD))$ by
\begin{align*}
\Phi_K \coloneqq (f-f_K)\indicator_K.
\end{align*}
Its Bochner integral vanishes in the sense that
\begin{align*}
\int_K\Phi_K(x)\,\mathrm d\mu_K(x) = \left(\frac{1}{|K|}\int_K(f-f_K)\dx\right)\indicator_K =0.
\end{align*}
Moreover, its Monte Carlo approximation is precisely the local stochastic error,
\begin{align*}
\frac{1}{N_K}\sum_{i=1}^{N_K}\Phi_K(X_i^K) = (\hat f_K-f_K)\indicator_K = \hat g_K\indicator_K.
\end{align*}
The mean-square identity for Monte Carlo estimators of Hilbert-space-valued random variables, cf.~\cite[Lem.~4.1]{BarthLangSchwab13}, therefore yields
\begin{align*}
\mathbb E\big[\lVert \hat g_K\indicator_K\rVert_{H^{-1}(\cD)}^2\big] &= \frac{1}{N_K}\int_K \lVert\Phi_K(x)\rVert_{H^{-1}(\cD)}^2 \,\mathrm d\mu_K(x) = \frac{\lVert\indicator_K\rVert_{H^{-1}(\cD)}^2}{N_K|K|} \lVert f-\Pi_0f\rVert_{L^2(K)}^2.
\end{align*}
The independence of the samples on distinct cells then implies
\begin{align*}
\mathbb E\bigg[ \bigg\lVert\sum_{K\in\tria}\hat g_K\indicator_K \bigg\rVert_{H^{-1}(\cD)}^2 \bigg] = \sum_{K\in\tria} \mathbb E\big[ \lVert\hat g_K\indicator_K\rVert_{H^{-1}(\cD)}^2 \big].
\end{align*}
This is exactly the diagonalization in \eqref{eq:DiagTerms}--\eqref{eq:Pasdafgwqrw}.
The abstract Hilbert-space framework thus explains the probabilistic orthogonality and the factor $N_K^{-1}$. The problem-specific ingredient of the present construction is the choice of the spatial reconstruction $\indicator_K$. In contrast, direct Monte Carlo quadrature of the load would associate the sampled value $f(x)$ with the Dirac distribution $\delta_x$. For $d\geq2$, however, $\delta_x\notin H^{-1}(\cD)$, so that the corresponding integrand $x\mapsto f(x)\delta_x$ is not $H^{-1}(\cD)$-valued. Spreading the sampled value over the cell through $\indicator_K\in L^2(\cD)\subset H^{-1}(\cD)$ avoids this obstruction. Finally, Lemma~\ref{lem:indicator} provides the problem-specific localization
\begin{align*}
\lVert\indicator_K\rVert_{H^{-1}(\cD)} \leq h_K\vartheta(h_K)|K|^{1/2}.
\end{align*}
This turns the abstract variance identity into the local, mesh-scaled data-oscillation estimate of Lemma~\ref{lem:Hm1}.
\end{remark}}
We summarize our findings in the following theorem.
\begin{theorem}[Approximation properties of $\hat \Pi_0$]\label{thm:MainPi0}
The randomized projection operator $\hat{\Pi}_0 \colon L^2(\cD) \to \mathbb{P}_0(\tria)$ defined in \eqref{eq:LowOrderRI} satisfies for all $f\in L^2(\cD)$
\begin{align*}
\mathbb E \big[ \lVert f - \hat \Pi_0 f\rVert_{L^2(\cD)}^2\big] \js{=} \sum_{K\in\tria} \left(1 + \frac{1}{N_K} \right) \lVert f - \Pi_0 f\rVert_{L^2(K)}^2
\end{align*}
and,  with $\vartheta$ defined in \eqref{eq:DefVartheta}, 
\begin{align*}
\mathbb E \big[ \lVert f - \hat \Pi_0 f\rVert_{H^{-1}(\cD)}^2\big] \leq \sum_{K\in\tria} h_K^2 \left( \pi^{-2} + \frac{\vartheta(h_K)^2}{N_K}\right) \lVert  f-\Pi_0 f \rVert^2_{L^2(K)}.
\end{align*}
\end{theorem}
\begin{proof}
This theorem summarizes the statements of Lemma~\ref{lem:ApxPi0L2} and Lemma~\ref{lem:Hm1}.
\end{proof}
We conclude this section with a note on the practical computation of $\hat{\Pi}_0$.
\begin{remark}[Computation of uniformly distributed random variables on simplices]
Our design of randomized projection operators requires uniformly distributed random variables $X_i^K \sim \mathcal U(K)$ for each simplex $K\in\mathcal T$. It is convenient to generate these samples on the reference simplex $K_\textup{ref}$ and then map them affinely to $K$. Below, we suggest two approaches for obtaining suitable random variables.

\textbf{Reflection trick.} Let $d=2$. We draw two i.i.d.\ random variables $U,V \sim \mathcal U(0,1)$. If the point $(U,V)$ is within the reference simplex $K_\textup{ref}$, we keep it. Otherwise, we reflect $(U,V)$ through the point $(1/2,1/2)$; that is,
\begin{align*}
X \coloneqq
\begin{cases}  (U,V) &\text{for } U+V\leq 1,\\
(1-U,\,1-V) &\text{for } U+V>1.
\end{cases}
\end{align*}
The resulting point $X \sim \mathcal U(K_\textup{ref})$ is uniformly distributed. Modified approaches extend to higher dimensions when one exploits the Kuhn partition of the (hyper-) cube discussed for example in~\cite[Sec.~4.1]{Bey00} or \cite[Rem.~2.2]{DieningStornTscherpel23}.

\textbf{Dirichlet sampling.} Let $d=2$. We draw independent random variables $U_0,U_1,U_2\sim \mathcal U(0,1)$ and set
\begin{align*}
\lambda_i \coloneqq \frac{\ln(U_i)}{\ln(U_0)+\ln(U_1)+\ln(U_2)}\qquad \text{for all }i=0,1,2.
\end{align*}
Due to the exponential distribution property $-\ln(U_i) \sim \textup{Exp}(1) = \textup{Gamma}(1,1)$ for all $i=0,1,2$, the random variable $(\lambda_0,\lambda_1,\lambda_2)$ has Dirichlet distribution \cite[Sec.~4.3.1]{KroeseTaimreBotev13} 
\begin{align*}
(\lambda_0,\lambda_1,\lambda_2) \sim \textup{Dirichlet}(1,1,1).
\end{align*}
Consequently, we obtain a uniformly distributed random variable $X \sim \mathcal{U}(K_\textup{ref})$ by
\begin{align*}
X \coloneqq \lambda_0 \begin{pmatrix}
0\\
0
\end{pmatrix} + \lambda_1 \begin{pmatrix}
1\\
0
\end{pmatrix} + \lambda_2 \begin{pmatrix}
0\\
1
\end{pmatrix} = \begin{pmatrix}
\lambda_1\\ \lambda_2
\end{pmatrix}.
\end{align*}
The construction extends verbatim to $d$-simplices by using $d+1$ independent exponentials and normalizing their sum.
\end{remark}

\section{Higher-order projection}\label{sec:HigherOrder}
If the right-hand side $f\in L^2(\cD)$ has some additional smoothness, it is beneficial to approximate the right-hand side by polynomials of higher degree $k\in \mathbb{N}$; that is, in the piecewise polynomial spaces 
\begin{align*}
\mathbb{P}_k(K) &\coloneqq \lbrace p \colon K \to \mathbb{R}\colon p\text{ is polynomial of maximal degree }k\rbrace\quad \text{for }K\in \tria,\\
\mathbb{P}_k(\tria) & \coloneqq \lbrace q \in L^2(\cD)\colon q|_K\in \mathbb{P}_k(K)\text{ for all }K\in \tria\rbrace.
\end{align*}  
Similar to Theorem~\ref{thm:MainPi0}, we aim at designing a randomized projection operator $\tilde{\Pi}_k \colon L^2(\cD) \to L^2(\Omega;\mathbb{P}_k(\tria))$ that is quasi-optimal in the sense that for suitable norms $\lVert \bigcdot \rVert$ and $\normm{\bigcdot}$ and $f\in L^2(\cD)$ 
\begin{align*}
\mathbb{E}\big[ \lVert f  - \tilde{\Pi}_k f\rVert^2 \big] \lesssim \min_{p\in \mathbb{P}_k(\tria)} \normm{ f - p }^2.
\end{align*}
Since the right-hand side equals zero for piecewise polynomials $f\in \mathbb{P}_k(\tria)$, this enforces the use of an operator that is a projection onto the space $\mathbb{P}_k(\tria)$. 
Typical deterministic interpolation operators such as the Scott--Zhang operator \cite{ScottZhang90} use orthogonal basis functions as weights to obtain such a projection property. However, evaluating the resulting integrals via some stochastic quadrature rule exactly for polynomials is challenging. We therefore exploit an alternative idea, known as discrete least-squares polynomial approximation \cite{CohenDavenportLeviatan13,CohenGiovanni17,MiglioratiNobileTempone15,MiglioratiNobileSchwerinTempone14}. 
Given $M_K\in \mathbb{N}$ independent and uniformly distributed random variables $Y_1^K,\dots,Y_{M_K}^K \sim \mathcal{U}(K)$ for all $K\in \tria$ and a fixed polynomial degree $k\in \mathbb{N}$, we set for all $f\in L^2(\cD)$ and $K\in \tria$ the random variable $\tilde{\Pi}_k f \in L^2(\Omega;\mathbb{P}_k(\tria))$ by
\begin{align}\label{eq:defHatfk}
(\tilde{\Pi}_k f)|_K \coloneqq \tilde{f}^k_K\qquad\text{with } \tilde{f}^k_K \in \argmin_{p\in\mathbb P_k(K)} \frac{1}{M_K} \sum_{i=1}^{M_K} |f(Y_i^K)-p(Y_i^K)|^2.
\end{align}
Before we show that the operator $\tilde{\Pi}_k\colon L^2(\cD) \to L^2(\Omega;\mathbb{P}_k(\tria))$ is almost surely well-defined in the sense that the minimizers in \eqref{eq:defHatfk} are unique, we introduce some notation needed throughout this section: 
Let  $(\psi_{K,1},\dots,\psi_{K,m})$ with $m = \dim \mathbb{P}_k(K)$ and $K\in \tria$ denote an orthonormal basis of $\mathbb P_k(K)$ in the sense that
\begin{align}\label{eq:orthoBasis}
\frac{1}{|K|}\int_K \psi_{K,i}\psi_{K,j}\dx=\delta_{ij}\qquad\text{for all }i,j=1,\dots,m.
\end{align}
By $\psi_K$ we abbreviate the vector $\psi_K \coloneqq (\psi_{K,1},\dots,\psi_{K,m})^\top$.
Moreover, let $\tilde{G}_K \in L^2(\Omega;\mathbb{R}^{m\times m})$ denote the empirical Gram matrix
\begin{align*}
\tilde G_K \coloneqq \frac{1}{M_K}\sum_{i=1}^{M_K} \psi_K(Y_i^K)\,\psi_K(Y_i^K)^\top.
\end{align*}
Finally, we set the random vector 
\begin{align*}
\tilde{\mathbf{f}}_K \coloneqq \frac{1}{M_K}\sum_{i=1}^{M_K} f(Y_i^K) \psi_K(Y_i^K).
\end{align*}
\begin{remark}[Choice of local basis functions]
The use of an orthonormal basis is advantageous for the analysis below, but it is not required for practical computations.
\end{remark}
\begin{lemma}[Well-definedness]\label{lem:wellDefiniedness}
Let $f\in L^2(\cD)$.
The coefficients $\tilde{z}_K\in L^2(\Omega;\mathbb{R}^m)$ of a minimizer $\tilde{f}^k_K = \tilde{z}_K \cdot \psi_K$ defined in \eqref{eq:defHatfk} satisfy
\begin{align*}
\tilde G_K \tilde z_K = \tilde{\mathbf{f}}_K.
\end{align*}
If $M_K \geq m \coloneqq \dim \mathbb{P}_k(K)$ for all $K\in \tria$, the empirical Gram matrix $\tilde G_K $ is almost surely symmetric positive definite, implying the uniqueness of the minimizer 
\begin{align*}
\tilde{f}^k_K = (\tilde{G}_K^{-1} \tilde{\mathbf{f}}_K) \cdot \psi_K .
\end{align*}
\end{lemma}
\begin{proof}
Let $f\in L^2(\cD)$ and $K\in \tria$.
We rewrite the local minimization in \eqref{eq:defHatfk} in terms of seeking coefficients $\tilde z_K \in L^2(\Omega;\mathbb{R}^m)$ with $\tilde{f}_K^k =\tilde  z_K \cdot \psi_K$ such that
\begin{align*}
\tilde z_K \in \argmin_{z\in \mathbb{R}^m} \mathcal{J}_K(z)\quad\text{with }\mathcal{J}_K(z) \coloneqq \frac1{2M_K} \sum_{i=1}^{M_K} \big| f(Y_i^K)-z \cdot \psi_K(Y_i^K)\big|^2.
\end{align*}
The first-order optimality condition reads
\begin{align*}
0&=\nabla \mathcal J_K(\tilde z_K)
= -\frac{1}{M_K}\sum_{i=1}^{M_K}\big(f(Y_i^K)-\tilde  z_K \cdot \psi_K(Y_i^K)\big)\psi_K(Y_i^K)\\
&= -\Big(\frac{1}{M_K}\sum_{i=1}^{M_K} f(Y_i^K)\,\psi_K(Y_i^K)\Big)
 + \Big(\frac{1}{M_K}\sum_{i=1}^{M_K}\psi_K(Y_i^K)\psi_K(Y_i^K)^\top\Big)\tilde  z_K.
\end{align*}
Consequently, the minimizer $\tilde z_K \in L^2(\Omega;\mathbb{R}^m)$ solves the normal equation
\begin{align*}
\tilde G_K \tilde z_K = \tilde{\mathbf{f}}_K \qquad\text{with}\qquad\tilde{\mathbf{f}}_K \coloneqq \frac{1}{M_K}\sum_{i=1}^{M_K} f(Y_i^K) \psi_K(Y_i^K).
\end{align*}
This verifies the first statement of the lemma.

Let $M_K \geq m \coloneqq \dim \mathbb{P}_k(K)$ and define for all $y=(y_1,\dots,y_m)\in K^m$ the function
\begin{align*}
D(y) \coloneqq D(y_1,\dots,y_m) \coloneqq \det\begin{pmatrix}
\psi_K(y_1)^\top\\
\vdots\\
\psi_K(y_m)^\top
\end{pmatrix}.
\end{align*}
The function $D\colon K^m \to \mathbb{R}$ is a polynomial. Since there exist unisolvent nodes $\xi_1,\dots,\xi_m\in K$ such as  Lagrange nodes, the vectors $\psi_K(\xi_1),\dots,\psi_K(\xi_m) \in \mathbb{R}^m$ are linearly independent, implying $D(\xi_1,\dots,\xi_m) \neq 0$. Consequently, the polynomial $D$ is not trivial; that is, $D\neq 0$, and thus its roots form a set of measure zero in $K^m$. Therefore the uniformly distributed random variable $Y^K = (Y^K_1,\dots,Y_m^K) \sim \mathcal{U}(K^m)$ satisfies
\begin{align*}
\mathbb{P}\big(D(Y^K) = 0\big) = 0.
\end{align*}
This shows that the first $m$ vectors $\psi_{K}(Y_1^K),\dots,\psi_{K}(Y_m^K)$ are almost surely linearly independent.
If they are linearly independent, any  $z \in \mathbb{R}^m \setminus \lbrace 0 \rbrace$ satisfies
\begin{align*}
z^\top \tilde G_K z = \frac{1}{M_K} \sum_{i=1}^{M_K} \big(z \cdot \psi_{K}(Y_i^K)\big)^2 > 0. 
\end{align*} 
This verifies that $\tilde{G}_K$ is almost surely symmetric positive definite.
\end{proof}
After clarifying that $\tilde \Pi_k$ is almost surely well-defined, we proceed with analyzing its approximation properties. Our analysis involves the $L^2$ orthogonal projection $\Pi_k\colon L^2(\cD) \to \mathbb{P}_k(\tria)$ defined via
\begin{align}\label{eq:defPik}
\int_\cD (\Pi_k f) p\dx = \int_\cD f p\dx\qquad\text{for all }f\in L^2(\cD) \text{ and }p\in \mathbb{P}_k(\tria).
\end{align} 
\begin{lemma}[Approximation properties in $L^2(\cD)$]\label{lem:L2Estimate}
Let $K\in \tria$, $k > 0$, and $M_K \geq m \coloneqq \dim \mathbb{P}_k(K)$. The operator $\tilde{\Pi}_k\colon L^2(\cD) \to L^2(\Omega;\mathbb{P}_k(\tria))$ defined in \eqref{eq:defHatfk} almost surely preserves polynomial right-hand sides in the sense that $(\tilde \Pi_k f)|_K  = f|_K$ for all $f|_K\in \mathbb{P}_k(K)$. For general $f\in L^2(\cD)$ one has
\begin{align*}
\mathbb{E}\big[\lVert f - \tilde{\Pi}_k f \rVert_{L^2(K)}^2\big] = \lVert f - \Pi_k f \rVert_{L^2(K)}^2 + \mathbb{E}\big[\lVert \tilde{\Pi}_k f - \Pi_k f \rVert_{L^2(K)}^2\big].
\end{align*}
Moreover, assuming the additional regularity $f|_K\in L^p(K)$ with $p \in (2,\infty)$ and having sufficiently many samples in the sense that $M_K \geq 4k (2(1-2/p)^{-1} + m) + 1$, one has with some constant $C_p = C(k,p,d) <\infty$ depending on  $k,p,d$ the estimate
\begin{align*}
\mathbb{E}\big[\lVert\tilde{\Pi}_k f - \Pi_k f \rVert_{L^2(K)}^2\big] \leq  C_p \frac{|K|^{1-2/p}}{M_K} \lVert f - \Pi_k f\rVert_{L^p(K)}^{2}.
\end{align*}
If $f\in L^p(\cD)$ for some $p\in (2,\infty)$ and $M_K \geq 4k (2(1-2/p)^{-1} + m) + 1$ for all $K\in \tria$, we have the global bound
\begin{align*}
\mathbb{E}\big[\lVert\tilde{\Pi}_k f - \Pi_k f \rVert_{L^2(\cD)}^2\big] \leq C_p \,|\cD|^{1-2/p}   \left( \sum_{K\in\tria} \frac{1}{M^{p/2}_K}\lVert f - \Pi_k f\rVert_{L^p(K)}^p \right)^{2/p}.
\end{align*} 
\end{lemma}
\begin{remark}[Best-approximation in $L^p$]
It follows by the $L^p$ stability of $\Pi_k$ and its projection property that for any $f\in L^p(\cD)$ and $K\in \tria$ the error $\lVert f - \Pi_k f\rVert_{L^p(K)}$ is bounded from above by the best-approximation error \cite{DouglasDupontWahlbin74}. More precisely, there exists some constant $C < \infty$ depending on $p \in [1,\infty]$, $k\in \mathbb{N}$, and the shape regularity of $K$ such that  
\begin{align*}
\lVert f - \Pi_k f\rVert_{L^p(K)} \leq C \min_{p\in \mathbb{P}_k(K)} \lVert f - p \rVert_{L^p(K)}.
\end{align*} 
\end{remark}
The proof of the lemma uses a constant $\Lambda_k< \infty$ defined in the following lemma.
\begin{lemma}[Christoffel-type quantity]\label{lem:ChristoffelTypeQuant}
For any $K\in \tria$ and basis $\psi_{K,1},\dots,\psi_{K,m} \in \mathbb{P}_k(K)$ with $k\in \mathbb{N}$ and \eqref{eq:orthoBasis} one has the identity
\begin{align*}
\Lambda_{k}\coloneqq \sup_{x\in K}\sum_{j=1}^m \psi_{K,j}(x)^2 = |K|\,  \sup_{x\in K} \sup_{p\in \mathbb{P}_k(K)\setminus \lbrace 0 \rbrace} \frac{p(x)^2}{\lVert p \rVert^2_{L^2(K)}} < \infty.
\end{align*}
The value $\Lambda_{k}$ depends on the polynomial degree $k$ and the dimensions $d$, but is independent of the simplex $K \in \tria$, its diameter $h_K$, and the choice of orthonormal basis functions $\psi_{K,1},\dots,\psi_{K,m} \in \mathbb{P}_k(K)$ satisfying \eqref{eq:orthoBasis}.
\end{lemma}
\begin{proof}
Let $x\in K$ and $p\in \mathbb{P}_k(K)\setminus \lbrace 0 \rbrace$ with $p = \sum_{j=1}^m p_j \psi_{K,j}$ and coefficients $(p_j)_{j=1}^m \subset \mathbb{R}$. Due to the orthonormality \eqref{eq:orthoBasis} the $L^2(K)$ norm of $p$ reads
\begin{align}\label{eq:Parsival}
\lVert p \rVert^2_{L^2(K)} = \sum_{j=1}^m p_j^2 \int_K \psi^2_{K,j} \dx = |K|\,\sum_{j=1}^m p_j^2 .
\end{align} 
Moreover, its value in any $x\in K$ reads
\begin{align*}
|p(x)| = \Big| \sum_{j=1}^m p_j \psi_{K,j}(x)\Big| \leq \Big(\sum_{j=1}^m p_j^2 \Big)^{1/2}\Big(\sum_{j=1}^m \psi_{K,j}(x)^2 \Big)^{1/2}.
\end{align*}
This yields an upper bound. The fact that the Cauchy--Schwarz inequality is sharp if $p_j = \psi_{K,j}(x)$ for all $j=1,\dots,m$, leads to a lower bound and thus verifies the identity. Consequently, the constant $\Lambda_k$ is independent of the chosen basis. We obtain the independence of $\Lambda_{k}$ from the underlying simplex $K$ by mapping orthonormal basis function on the reference element to $K$.
\end{proof}
\begin{proof}[Proof of Lemma~\ref{lem:L2Estimate}]
Our definition in \eqref{eq:defHatfk} verifies in combination with the almost sure uniqueness of the minimizer discussed in Lemma~\ref{lem:wellDefiniedness} the preservation of polynomials. The Pythagorean theorem yields the first identity in the lemma. It remains to verify the bound for general $f\in L^p(\cD)$ with $p>2$ and $K\in \tria$.

\textit{Step 1 (H\"older estimate)}. 
Since the mapping in \eqref{eq:defHatfk} is linear and preserves polynomials such as $\Pi_k f$, we obtain with $\tilde{\eta}_K \coloneqq M^{-1}_K \sum_{i=1}^{M_K} \big( f- \Pi_k f\big)(Y_i^K)\, \psi_K(Y_i^K)$ almost surely the identity
\begin{align}\label{eq:ProofExt}
\tilde g_K \coloneqq (\tilde{\Pi}_k f - \Pi_k f)|_K = \big(\tilde G_K^{-1} \tilde \eta_K\big) \cdot \psi_K.
\end{align}
Hence, the approximation error is due to \eqref{eq:Parsival} almost surely bounded by 
\begin{align}\label{eq:ProofTemp222}
\lVert \tilde{\Pi}_k f - \Pi_k f\rVert^2_{L^2(K)} = \lVert \tilde G_K^{-1} \tilde \eta_K \rVert^2_{\ell^2}\, | K | \leq \lVert \tilde G_K^{-1} \rVert^2\, \lVert \tilde \eta_K\rVert^2_{\ell^2} |K|.
\end{align}
Let $c = c(d,k) \in (1, \infty)$ denote the constant from Corollary~\ref{cor:NrobustBound} in the appendix.
We define for all $j\in \mathbb{N}$ the events 
\begin{align*}
\omega_K^0 \coloneqq \lbrace 0 \leq \lVert \tilde G_K^{-1} \rVert < c \rbrace\qquad \text{and}\qquad \omega_K^j \coloneqq \lbrace c^{j} \leq \lVert \tilde G_K^{-1} \rVert < c^{j+1} \rbrace.
\end{align*}
Let $j\in \mathbb{N}_0$ be fixed. We have by definition
\begin{align*}
\lVert \tilde G_K^{-1} \tilde \eta_K \rVert_{\ell^2}^2\indicator_{\omega^j_K} \leq c^{2j+2}\,\lVert\tilde  \eta_K\rVert_{\ell^2}^2\indicator_{\omega^j_K}.
\end{align*}
Combining \eqref{eq:ProofTemp222} with this bound implies
\begin{align*}
\mathbb E\big[\lVert \tilde g_K\rVert_{L^2(K)}^2\,\indicator_{\omega^j_K}\big] \leq c^{2j+2}\,|K|\,\mathbb E \big[ \lVert \tilde \eta_K\rVert_{\ell^2}^2 \indicator_{\omega^j_K}\big].
\end{align*}
Hölder's inequality yields for any $p>2$ and $\gamma\coloneqq 1-2/p\in (0,1)$ that
\begin{align*}
\mathbb E \left[\lVert \tilde \eta_K\rVert_{\ell^2}^2\,\indicator_{\omega_K^j}\right] \leq \mathbb E[\lVert \tilde  \eta_K\rVert_{\ell^2}^{p}]^{2/p}\,\mathbb P(\omega_K^j)^{\gamma}.
\end{align*}
Hence, we obtain for all $j\in \mathbb{N}$ the bound
\begin{align}\label{eq:Prooasdsa}
\mathbb E\left[\lVert\tilde  g_K\rVert_{L^2(K)}^2\,\indicator_{\omega_K^j}\right] \leq c^2\, |K|\,c^{2j}\,\big(\mathbb E[\lVert\tilde  \eta_K\rVert_{\ell^2}^{p}]\big)^{2/p}\, \mathbb P(\omega_K^j)^{\gamma}.
\end{align}

\textit{Step 2 (Bound for $\mathbb E\big[\lVert\tilde \eta_K\rVert_{\ell^2}^{p}\big]^{2/p}$).}
Set the difference 
\begin{align*}
\tilde \delta_i \coloneqq \big( f- \Pi_k f\big)(Y_i^K)\, \psi_K(Y_i^K)\qquad\text{for all }i=1,\dots,M_K.
\end{align*}
Since $f - \Pi_k f$ is orthogonal onto piecewise polynomials (including $\psi_{K,j}$), we have
\begin{align*}
\mathbb{E}[\tilde \delta_i ] =\frac1{|K|} \int_K \big( f- \Pi_k f\big)\, \psi_K\dx  = 0\qquad\text{for all }i=1,\dots,M_K.
\end{align*}
Moreover, its $p$-th moment is bounded according to Lemma~\ref{lem:ChristoffelTypeQuant} by 
\begin{align*}
\mathbb{E}\big[\lVert \tilde \delta_i \rVert_{\ell^2}^p \big] \leq \mathbb{E}\big[ | (f - \Pi_kf)(Y_K^i) |^p \lVert \psi_K(Y_i^K) \rVert_{\ell^2}^p\big] \leq \Lambda_k^{p/2} |K|^{-1} \lVert f - \Pi_kf\rVert^p_{L^p(K)} < \infty.
\end{align*}
Hence, we can apply the Marcinkiewicz--Zygmund inequality \cite[Sec.~10.3 Thm.~2]{ChowTeicher97}, which yields in combination with the identity $\tilde{\eta}_K  = M_K^{-1} \sum_{i=1}^{M_K}\tilde  \delta_i$, and H\"older's inequality for vectors the existence of some constant $C_p< \infty$ depending solely on $p$ such that
\begin{align*}
\mathbb E\big[\lVert\tilde \eta_K\rVert_{\ell^2}^{p}\big] & = M_K^{-p} \mathbb E\Big[\Big\lVert \sum_{i=1}^{M_K}\tilde \delta_i\Big\rVert_{\ell^2}^{p}\Big] \leq C_p M_K^{-p} \mathbb E\Big[ \Big(\sum_{i=1}^{M_K}\lVert\tilde \delta_i\rVert^2_{\ell^2}\Big)^{p/2}\Big]\\
& \leq C_p M_K^{-p} \mathbb E\Big[ M_K^{p/2-1} \sum_{i=1}^{M_K}\lVert\tilde \delta_i\rVert^p_{\ell^2} \Big]  =  C_p M_K^{-p/2}\mathbb{E}\big[ \lVert\tilde  \delta_1\rVert^p_{\ell^2}\big] .
\end{align*}
Using the bound $\mathbb{E}\big[ \lVert \tilde \delta_1\rVert^p_{\ell^2}\big] \leq  \Lambda_k^{p/2} |K|^{-1} \lVert f - \Pi_k f\rVert_{L^p(K)}^p$ from above, we obtain
\begin{align}\label{eq:Proofds1}
\mathbb E\big[\lVert\tilde \eta_K\rVert_{\ell^2}^{p}\big]^{2/p} \leq C_p^{2/p}\frac{\Lambda_k}{M_K}  |K|^{-2/p} \lVert f - \Pi_k f\rVert_{L^p(K)}^{2}.
\end{align}

\textit{Step 3 (Bound for the second addend $ \mathbb P(\omega_K^j)$)}. 
Assume that $M_K \geq 4k (2/\gamma + m) + 1$.
Corollary~\ref{cor:NrobustBound} displayed in the appendix yields the existence of a constant $C = C(d,k,M_{K,0})$ with $M_K \geq M_{K,0} \coloneqq \lceil 4k (2/\gamma + m) + 1\rceil$ such that
\begin{align*}
\mathbb{P}\big(t \leq \lVert\tilde  G_K^{-1}\rVert \big) \leq C \Lambda_k^m t^{-\left(\frac{M_{K,0}}{4k}-m\right)} \leq  C \Lambda_k^m t^{-\left(\frac2\gamma + \frac{1}{4k}\right)}\qquad\text{for all }t \geq c \geq 1.
\end{align*}
The definition of $\omega_K^j$ thus leads to the bound
\begin{align}\label{eq:Proofds2}
\mathbb{P}(\omega_K^j)^\gamma \leq C^\gamma \Lambda_k^{\gamma m} c^{-j \left(2 + \frac{\gamma}{4k}\right)}\qquad\text{for all }j\in \mathbb{N}_0.
\end{align}

\textit{Step 4 (Combining the results).}
Applying the bound in  \eqref{eq:Proofds1}--\eqref{eq:Proofds2} to \eqref{eq:Prooasdsa} yields with some constant $C = C(k,d,p)$ the estimate
\begin{align*}
\mathbb E\left[\lVert\tilde  g_K\rVert_{L^2(K)}^2\,\indicator_{\omega_K^j}\right] \leq C \js{M_K^{-1}}\, |K|^{1-2/p}\,c^{-j\gamma /(4k)}  \lVert f - \Pi_k f\rVert_{L^p(K)}^{2}\quad\text{for all }j \in \mathbb{N}_0.
\end{align*}
Summing over all $j\in \mathbb{N}_0$ leads to 
\begin{align*}
\mathbb E\left[\lVert\tilde  g_K\rVert_{L^2(K)}^2\right]& = \sum_{j=0}^\infty \mathbb E\left[\lVert\tilde  g_K\rVert_{L^2(K)}^2\,\indicator_{\omega_K^j}\right]\\
& \leq C\js{M_K^{-1}}|K|^{1-2/p} \lVert f - \Pi_k f\rVert_{L^p(K)}^{2} \sum_{j=0}^\infty c^{-j\gamma/(4k)}.
\end{align*}
The sum on the right-hand side is finite, resulting in the lemma's local bound.

\textit{Step 5 (Global bound)}.
Suppose that $f\in L^p(\cD)$ for some $p>2$. Moreover, assume that for all $K\in \tria$ the number of samples satisfies $M_K \geq 4k (2/\gamma + m) + 1$. Then H\"older's inequality with $q = p/2$ and $q' = p/(p-2)$ and the local bound yield
\begin{align*}
\mathbb{E}\big[\lVert\tilde{\Pi}_k f - \Pi_k f \rVert_{L^2(\cD )}^2\big]&= \sum_{K\in \tria} \mathbb{E}\big[\lVert\tilde{\Pi}_k f - \Pi_k f \rVert_{L^2(K)}^2\big]\\
&\leq 
C_p \sum_{K\in\tria} \frac{|K|^{1-2/p}}{M_K} \lVert f - \Pi_k f\rVert_{L^p(K)}^2 \\
&\leq C_p\,  |\cD|^{1-2/p}   \Big( \sum_{K\in\tria} \frac{1}{M^{p/2}_K}\lVert f - \Pi_k f\rVert_{L^p(K)}^p \Big)^{2/p}.\qedhere
\end{align*}
\end{proof}
As shown in \eqref{eq:ProofExt}, the local contributions $\tilde g_K =  (\tilde{G}_K^{-1} \tilde{\eta}_K) \cdot \psi_K$ result from the composition of a random vector $\tilde{\eta}_K$ and a random matrix $\tilde{G}_K^{-1}$ which are dependent. This has the severe drawback that the difference $\tilde g_K$ is -- even in its first moment -- biased, as illustrated in the following example.
\begin{example}[Counterexample]\label{ex:counter_bias_mean}
Let $d=1$, $K=(0,1) = \cD$, $k=1$, $M_K=2$, and $m=\dim\mathbb{P}_1(K)=2$.
Let $f(x) = x^2$ and define its approximation $\tilde{f}$ as in \eqref{eq:defHatfk} in the sense that with random variables $Y_1,Y_2 \sim \mathcal U(0,1)$ we have
\begin{align*}
\tilde f\in\argmin_{p\in\mathbb{P}_1(K)}\frac12\sum_{i=1}^2 |f(Y_i)-p(Y_i)|^2.
\end{align*}
The minimizer $\tilde{f}$ interpolates the two data points
\begin{align*}
\tilde f(Y_i)=f(Y_i)=Y_i^2\qquad\text{with } i=1,2.
\end{align*}
Hence, it reads $\tilde f(x)=\tilde ax+\tilde b$ with constant $\tilde{a},\tilde{b}$ determined by
\begin{align*}
\tilde aY_1+\tilde b=Y_1^2\qquad \text{and}\qquad \tilde aY_2+\tilde b=Y_2^2.
\end{align*}
We obtain almost surely the representation
\begin{align*}
\tilde a=\frac{Y_1^2-Y_2^2}{Y_1-Y_2}=Y_1+Y_2
\qquad\text{and}\qquad
\tilde b=Y_1^2-\tilde aY_1=-Y_1Y_2.
\end{align*}
This yields the first moment
\begin{align*}
\int_0^1 \tilde f(x)\dx
=\int_0^1\big((Y_1+Y_2)x-Y_1Y_2\big)\dx
=\frac{Y_1+Y_2}{2}-Y_1Y_2.
\end{align*}
Taking expectations and using $\mathbb E[Y_1] = \mathbb E[Y_2]=1/2$ as well as the independence
$\mathbb E[Y_1Y_2]=\mathbb E[Y_1]\mathbb E[Y_2]=1/4$ yield
\begin{align*}
\mathbb E\left[\int_0^1 \tilde f (x)\dx\right]
=\frac{\mathbb E[Y_1]+\mathbb E[Y_2]}{2}-\mathbb E[Y_1Y_2] =\frac{1/2+1/2}{2}-\frac14 =\frac14.
\end{align*}
This shows that 
\begin{align*}
\mathbb E\left[\int_0^1 \tilde f (x)\dx\right] = \frac{1}{4} \neq \frac{1}{3}=\int_0^1 x^2\dx  = \int_0^1 f\dx  = \int_0^1 \Pi_0 f \dx.
\end{align*}
\end{example}
The unbiasedness in the first moment was a key in the diagonalization argument in Lemma~\ref{lem:Hm1}. This motivates a correction of the first moment similar to the Fortin trick, see for example \cite[Sec.~2.4]{DieningStornTscherpel22}. In particular, we use the low-order random projection operator defined in \eqref{eq:LowOrderRI} to compute the corrected randomized projection 
\begin{align}\label{eq:DefCorrectedOperator}
\hat{\tilde{\Pi}}_k f \coloneqq \tilde{\Pi}_k f - \hat \Pi_0 (\tilde{\Pi}_k f - f),
\end{align}
where the random variables $(X_i^K)_{i=1}^{N_K}$ used in the definition of $\hat{\Pi}_0$ are independent of the random variables $(Y_i^K)_{i=1}^{M_K}$ used in the definition of $\tilde{\Pi}_k$ for all $K\in \tria$.
This leads to the following result.
\begin{lemma}[Unbiased first moment]\label{lem:unbias}
For any $f\in L^2(\cD)$ and $K\in \tria$ one has
\begin{align*}
\mathbb{E}\Big[ \int_K \big(\hat{\tilde{\Pi}}_k f - f\big) \dx \Big] = \mathbb{E}\Big[ \int_K \big(\hat{\tilde{\Pi}}_k f - \Pi_k f\big) \dx \Big] = 0.
\end{align*}
\end{lemma}
\begin{proof}
Let $K\in\tria$ and $f\in L^2(\cD)$. We define the random variables $\tilde{\delta} \coloneqq f - \tilde{\Pi}_k f$ and $\hat{\tilde{\delta}}_K \coloneqq \big(\hat\Pi_0 \tilde{\delta}\big)|_K$; that is,
\begin{align*}
\hat{\tilde{\delta}}_K = \frac1{N_K}\sum_{i=1}^{N_K} \tilde{\delta}(X_i^K).
\end{align*}
Let $\mathcal{G}_K^Y \coloneqq \sigma(Y_1^K,\dots,Y_{M_K}^K)$ be the $\sigma$-algebra generated by the samples defining $\tilde f_K^k \coloneqq (\tilde{\Pi}_k f)|_K$. Then $\tilde f^k_K$ and hence $\tilde \delta_K$ are $\mathcal{G}^Y_K$-measurable, while $X_1^K,\dots,X_{N_K}^K$ are independent of $\mathcal{G}^Y_K$ by assumption.
Conditioning on $\mathcal{G}^Y_K$ leads to
\begin{align*}
\mathbb{E}\big[\hat{\tilde{\delta}}_K\mid \mathcal{G}^Y_K\big]& =\frac1{N_K}\sum_{i=1}^{N_K}\mathbb{E}\big[\tilde \delta(X_i^K)\mid \mathcal{G}^Y_K\big] = \frac1{|K|}\int_K \tilde\delta(x)\dx.
\end{align*}
Combining this equality with the tower property results in
\begin{align*}
\mathbb E\big[\hat{\tilde\delta}_K\big]  = \mathbb E\Big[\mathbb E\big[\hat{\tilde\delta}_K\mid \mathcal G^Y_K\big]\Big] = \frac1{|K|}\,\mathbb E\Big[\int_K (f - \tilde{\Pi}_k f)\dx\Big].
\end{align*}
This and the definition in \eqref{eq:DefCorrectedOperator}, which yields $ f|_K - (\hat{\tilde{\Pi}}_k f)|_K = \tilde{\delta}|_K - \hat{\tilde{\delta}}_K$, lead to
\begin{align*}
\mathbb E\Big[\int_K (f-\hat{\tilde{\Pi}}_k f)\dx\Big] = \mathbb E\Big[\int_K ( f - \tilde{\Pi}_kf)\dx\Big] - |K|\,\mathbb E\big[\hat{\tilde\delta}_K\big] =0.
\end{align*}
This yields the first statement of the lemma. The second one follows by the identity
\begin{align*}
&\int_K f \dx = \int_K \Pi_k f\dx.\qedhere
\end{align*}
\end{proof}
The unbiased first moment allows us to apply as in Lemma~\ref{lem:ApxPi0L2} a diagonalization argument, leading to the localization of the $H^{-1}(\cD)$ norm displayed in the following.
\begin{lemma}[Approximation properties in $H^{-1}(\cD)$]\label{lem:ApxHminHigherOrder}
Let $f\in L^2(\cD)$. Then the expected $H^{-1}(\cD)$ error splits into
\begin{align*}
\mathbb{E}\big[\lVert f-\hat{\tilde{\Pi}}_k f\rVert_{H^{-1}(\cD)}^2 \big] \leq 2\, \lVert f - \Pi_k f \rVert_{H^{-1}(\cD)}^2 + 2\, \mathbb{E}\big[\lVert \Pi_k f-\hat{\tilde{\Pi}}_k f\rVert_{H^{-1}(\cD)}^2 \big].
\end{align*}
The deterministic term is controlled by 
\begin{align*}
\lVert f - \Pi_k f \rVert_{H^{-1}(\cD)}^2 \leq \sum_{K\in \tria} \pi^{-2} h_K^2 \lVert f - \Pi_k f\rVert^2_{L^2(K)}. 
\end{align*}
The expectation satisfies, with constant $\vartheta(h_K)$ defined for all $K\in \tria$ in \eqref{eq:DefVartheta},
\js{\begin{align*}
\mathbb{E}\big[\lVert \Pi_k f-\hat{\tilde{\Pi}}_k f\rVert_{H^{-1}(\cD)}^2 \big] & \leq \sum_{K\in \tria} \pi^{-2} h_K^2 \mathbb{E}\big[ \lVert (1-\Pi_0) (\Pi_k f-{\tilde{\Pi}}_k f)\rVert_{L^2(K)}^2 \big]\\
&\quad  + \sum_{K\in \tria} h_K^2 \frac{\vartheta(h_K)^2}{N_K} \mathbb{E}\big[ \lVert (1-\Pi_0) (f-\tilde{\Pi}_k f) \rVert_{L^2(K)}^2\big].
\end{align*}}%
\end{lemma}
\begin{proof}
The decomposition of the error and the deterministic bound follow as in Lemma~\ref{lem:Hm1}. To obtain the lemma's remaining inequality for the expectation, we \js{split the expectation into 
\begin{align}\label{eq:Proofasdsadsadsdsa}
\begin{aligned}
\mathbb{E}\big[\lVert \Pi_k f-\hat{\tilde{\Pi}}_k f\rVert_{H^{-1}(\cD)}^2 \big] &\leq  \mathbb{E}\big[\lVert (1-\Pi_0) (\Pi_k f-\hat{\tilde{\Pi}}_k f)\rVert_{H^{-1}(\cD)}^2 \big]\\
&\quad +  \mathbb{E}\big[\lVert \Pi_0( \Pi_k f-\hat{\tilde{\Pi}}_k f)\rVert_{H^{-1}(\cD)}^2 \big].
\end{aligned}
\end{align}}
Using again the arguments from the proof of Lemma~\ref{lem:Hm1} and the property $(1-\Pi_0) \hat{\tilde{\Pi}}_k = (1-\Pi_0) \tilde{\Pi}_k$, we obtain the bound
\begin{align*}
\lVert (1-\Pi_0) (\Pi_k f-\hat{\tilde{\Pi}}_k f)\rVert_{H^{-1}(\cD)}^2 \leq \sum_{K\in \tria} \pi^{-2} h_K^2 \lVert (1-\Pi_0) (\Pi_k f-{\tilde{\Pi}}_k f)\rVert_{L^2(K)}^2.
\end{align*}
To bound the latter term in \eqref{eq:Proofasdsadsadsdsa}, we observe that with $\tilde{\delta}\coloneqq f- \tilde \Pi_k f$
\begin{align*}
f - \hat{\tilde{\Pi}}_k f = \tilde{\delta} - \hat{\Pi}_0  \tilde{\delta} \qquad\text{and}\qquad \Pi_0(f - \hat{\tilde{\Pi}}_k f) = \Pi_0 \tilde{\delta}  - \hat{\Pi}_0 \tilde \delta.
\end{align*}
Hence, Lemma~\ref{lem:Hm1} leads (using conditional expectation and the tower property as in the proof of Lemma~\ref{lem:unbias}) to the bound
\begin{align*}
\mathbb{E}\big[\lVert \Pi_0( \Pi_k f-\hat{\tilde{\Pi}}_k f)\rVert_{H^{-1}(\cD)}^2 \big] & = \mathbb{E}\big[\lVert \Pi_0\tilde \delta - \hat \Pi_0 \tilde \delta\rVert_{H^{-1}(\cD)}^2 \big]\\
& \leq \sum_{K\in \tria} h_K^2 \frac{\vartheta(h_K)^2}{N_K} \mathbb{E}\big[ \lVert (1-\Pi_0) \tilde{\delta} \rVert_{L^2(K)}^2\big]. \qedhere
\end{align*}
\end{proof}
We summarize this section's result in the following theorem.
\begin{theorem}[Approximation properties of $ \tilde \Pi_k$ and $\hat{\tilde{\Pi}}_k$]\label{thm:MainPik}
The randomized projection operator $\tilde \Pi_k$ defined in \eqref{eq:defHatfk} is an almost surely well-defined projection onto $\mathbb{P}_k(\tria)$. For any $f\in L^p(\cD)$ with $p\in (2,\infty)$ it satisfies, with a constant $C_p < \infty$ depending solely on $p,k,d$ and under the assumption that the number of samples satisfies $M_K \geq 4k(2(1-2/p)^{-1} + m) +1$ on each element $K\in \tria$, 
\begin{align*}
\mathbb{E}\big[\lVert f - \tilde{\Pi}_k f\rVert_{L^2(\cD)}^2 \big] \leq \lVert f - \Pi_k f\rVert_{L^2(\cD)}^2 + C_p \sum_{K\in \tria} \frac{|K|^{1-2/p}}{M_K} \lVert f - \Pi_k f\rVert_{L^p(K)}^2.
\end{align*}
Furthermore, its modification $\hat{\tilde{\Pi}}_k$ satisfies under the same assumptions
\begin{align*}
&\mathbb{E}\big[\lVert f - \hat{\tilde{\Pi}}_k f\rVert_{H^{-1}(\cD)}^2 \big]  \leq 2 \sum_{K\in \tria}h_K^2 
\left(\pi^{-2} + \frac{\vartheta(h_K)^2}{N_K}\right) \lVert f - \Pi_k f\rVert_{L^2(K)}^2\\
&\qquad\qquad\qquad + 2 \sum_{K\in \tria}h_K^2 \left( \pi^{-2}  + \frac{\vartheta(h_K)^2}{N_K}\right) \frac{C_p |K|^{1-2/p}}{M_K} \lVert f - \Pi_k f\rVert^2_{L^p(K)}  .
\end{align*}
\end{theorem}
\begin{proof}
This theorem combines Lemma~\ref{lem:wellDefiniedness}, \ref{lem:L2Estimate}, and \ref{lem:ApxHminHigherOrder}.
\end{proof}

\begin{remark}[Alternative finite element spaces]
We only discussed randomized mappings $\tilde{\Pi}$ onto piecewise polynomial spaces $\mathbb{P}_k(\tria)$. In some applications it might be beneficial to map into finite element spaces $V_h$ such as the Lagrange finite element space. In this case we can combine our randomized projection operator with a deterministic interpolation operator $\mathcal{I}\colon L^2(\cD) \to V_h$ in the sense that we define 
\begin{align*}
\tilde{\mathcal{I}} \coloneqq \mathcal{I} \circ \tilde{\Pi} \colon L^2(\cD) \to V_h.
\end{align*}
Under the assumption that $\mathcal{I}|_{\mathbb{P}_k(\tria)}$ can be evaluated exactly, as for example in the case of Scott--Zhang-type operators with polynomial weights, this randomized projection operator is computable for any right-hand side $f\in L^2(\cD)$. Moreover, a triangle inequality reveals for any norm $\lVert \bigcdot \rVert$ that
\begin{align*}
\lVert f - \tilde{\mathcal{I}} f\rVert \leq  \lVert f - \mathcal{I} f\rVert + \lVert \mathcal{I} \rVert \,  \lVert f - \tilde{\Pi} f\rVert.
\end{align*}
Hence, $L^2(\cD)$ stability as in \cite{ScottZhang90} or even $H^{-1}(\cD)$ stability as in \cite{DieningStornTscherpel21b} allows us to obtain a randomized projection operator onto $V_h$ with expected randomized approximation error in $L^2(\cD)$ or $H^{-1}(\cD)$ bounded by the interpolation error plus the (almost) optimal approximation error of $\tilde{\Pi}$.
\end{remark}

\section{Application: Smoothers for rough right-hand sides}\label{sec:Smoother}
In this section we illustrate the application of the randomized projection operators $\hat{\Pi} \in \lbrace \hat{\tilde \Pi}_k\colon k\in \mathbb{N}\rbrace \cup \lbrace \hat \Pi_0\rbrace$ as smoothers for right-hand sides of partial differential equations. The Poisson model problem serves as our prototypical example: Given a bounded Lipschitz domain $\cD\subset \mathbb{R}^d$ for $d\geq 2$ and a right-hand side $f\in L^2(\cD)$, we seek the solution $u\in H^1_0(\cD)$ to
\begin{align}\label{eq:Laplace}
\int_\cD \nabla u \cdot \nabla v\dx = \int_\cD f v\dx \qquad\text{for all }v\in H^1_0(\cD).
\end{align}
The finite element approximation $u_h \in V_h$ with finite dimensional space $V_h$ such as the Lagrange finite element space is defined by
\begin{align}\label{eq:PmpVh}
\int_\cD \nabla u_h\cdot \nabla v_h \dx = \int_\cD f v_h\dx \qquad\text{for all }v_h \in V_h.
\end{align}
Instead of solving the discretized problem in \eqref{eq:PmpVh}, we seek $\hat{u}_h \in L^2(\Omega;V_h)$ with 
\begin{align}\label{eq:PmpPiK}
\int_\cD \nabla \hat{u}_h\cdot \nabla v_h \dx = \int_\cD \hat{\Pi} f\, v_h\dx \qquad\text{for all }v_h \in V_h.
\end{align}
The need for our randomized projection operator as smoother is twofold. From a practical point of view, the load $f\in L^2(\cD)$ is often solely known in a finite number of (randomly placed) points per cell $K\in \tria$ due to finitely many measurements -- fitting directly in the framework of the randomized projection operator $\hat{\Pi}$. From a theoretical point of view, standard a priori error analysis yields with maximal mesh size $h\coloneqq \max_{T\in \tria} h_T$ and convex domains $\cD$ the error estimate
\begin{align*}
\lVert \nabla u - \nabla u_h \rVert_{L^2(\cD)} \lesssim h\, \lVert f \rVert_{L^2(\cD)}.
\end{align*}
However, this bound tacitly assumes that the load $\int_\cD f v_h \dx$ can be evaluated
exactly for all $v_h \in V_h$. In practice, $f$ must be approximated numerically. 
Classical deterministic quadrature rules for the computation of load vectors typically require additional smoothness of the integrand such as piecewise $W^{1,\infty}(K)$ regularity for all $K\in \tria$ used in \cite[Thm.~33.17]{ErnGuermond21b}, see \cite{Fix72,Ciarlet2002,BabuskaBanerjeeLi11} for further results.
In other words, deterministic quadrature requires additional smoothness of the right-hand side far beyond $f\in L^2(\cD)$. 
A remedy is (stratified) Monte Carlo quadrature. This approach leads to an approximation $\hat{u}_h \in L^2(\Omega;V_h)$ that satisfies with maximal mesh size $h> 0$ and $d= 2$ \cite[Thm.~3.5]{KrusePolydoridesWu19}
\begin{align*}
\mathbb{E}\big[\lVert \nabla u - \nabla \hat u_h \rVert_{L^2(\cD)}\big] \lesssim \begin{cases}
 h\, \lVert f \rVert_{L^2(\cD)} + h^{1-2/p} \lVert f \rVert_{L^p(\cD)}&\text{for }p\in [2,\infty),\\
  h\, \lVert f \rVert_{L^2(\cD)} + h \sqrt{\log(h^{-1})}\, \lVert f \rVert_{L^\infty(\cD)}&\text{for }p = \infty.
 \end{cases}
\end{align*}
Thus, randomized quadrature substantially relaxes the regularity requirements compared to deterministic rules and yields (almost) optimal rates for bounded data. Nevertheless, for merely square-integrable right-hand sides $f\notin L^{2+\varepsilon}(\cD)$ there is a significant loss of convergence induced by the quadrature error that we overcome with our smoother, as discussed in the following theorem. The theorem involves the auxiliary solution $\bar{u}_h \in V_h$ defined via the $L^2(\cD)$ orthogonal projection $\Pi_k$ onto $\mathbb{P}_k(\tria)$ by
\begin{align*}
\int_\cD \nabla \bar u_h\cdot \nabla v_h \dx = \int_\cD \Pi_k f v_h\dx \qquad\text{for all }v_h \in V_h.
\end{align*}
\begin{theorem}[Expected additional error]\label{thm:Main1}
Approximating the finite element solution $u_h \in V_h$ by $\bar{u}_h \in V_h$ leads to the additional (deterministic) error 
\begin{align*}
\lVert \nabla u_h - \nabla \bar u_h \rVert_{L^2(\cD)}^2 \leq \pi^{-2} \sum_{K\in \tria} h_K^2 \lVert f - \Pi_k f\rVert^2_{L^2(K)}.
\end{align*}
Approximating $\bar{u}_h \in V_h$ by $\hat{u}_h \in L^2(\Omega;V_h)$ leads to an additional expected error  
\begin{align*}
\mathbb{E}\big[ \lVert \nabla \bar u_h - \nabla \hat u_h \rVert_{L^2(\cD)}^2\big] \leq  \mathbb{E}\big[ \lVert \Pi_k f - \hat \Pi f \rVert_{H^{-1}(\cD)}^2\big].
\end{align*}
For $\hat{\Pi} = \hat{\Pi}_0$ and $k= 0$ the latter term is bounded by 
\begin{align*}
\mathbb{E}\big[ \lVert \Pi_0  f - \hat \Pi_0 f \rVert_{H^{-1}(\cD)}^2\big] \leq \sum_{K\in \tria} h_K^2 \frac{\vartheta(h_K)^2}{N_K} \lVert f - \Pi_0 f\rVert_{L^2(K)}^2.
\end{align*}
For $\hat{\Pi} = \hat{\tilde{\Pi}}_k$ and $k \in \mathbb{N}$ the latter term is, under the assumption that $f\in L^p(\cD)$ with $p\in (2,\infty)$ and the number of samples $M_K \geq 4k (2(1-2/p)^{-1} + m) + 1$  for all $K\in \tria$, bounded by 
\begin{align*}
&\mathbb{E}\big[\lVert \Pi_k f-\hat{\tilde{\Pi}}_k f\rVert_{H^{-1}(\cD)}^2 \big] \leq 2 \sum_{K\in \tria} h_K^2 \frac{\vartheta(h_K)^2}{N_K}  \lVert f-{\Pi}_k f \rVert_{L^2(K)}^2 \\
&\qquad\qquad\qquad+ 2\sum_{K\in \tria}
h_K^2 \Big(\pi^{-2 } + \frac{\vartheta(h_K)^2}{N_K} \Big) C_p \frac{|K|^{1-2/p}}{M_K}  \lVert f-{\Pi}_k f \rVert_{L^p(K)}^2.
\end{align*}
\end{theorem}
\begin{proof}
Poincare's inequality on convex domains yields for $\bar e_h \coloneqq  u_h - \bar u_h$
\begin{align*}
\lVert \nabla \bar e_h \rVert_{L^2(\cD)}^2 &= \int_\cD (f- \Pi_k f) (\bar e_h - \Pi_k \bar e_h)\dx\\
& \leq \sum_{K\in \tria } \lVert f - \Pi_k f \rVert_{L^2(K)} \pi^{-1} h_K \lVert  \nabla \bar e_h \rVert_{L^2(K)}\\
& \leq \Big(\pi^{-2} \sum_{K\in \tria } h_K^2 \lVert f - \Pi_k f\rVert_{L^2(K)}^2 \Big)^{1/2} \lVert \nabla\bar e_h \rVert_{L^2(\cD)}. 
\end{align*}
Similarly, we obtain for  $\hat{e}_h \coloneqq \bar u_h - \hat{u}_h$ the bound
\begin{align*}
\lVert \nabla \hat e_h \rVert_{L^2(\cD)}^2 = \int_\cD (\Pi_k f - \hat{{ \Pi}} f) \hat{e}_h \dx \leq \lVert \Pi_k f - \hat \Pi f \rVert_{H^{-1}(\cD)} \lVert \nabla \hat{e}_h \rVert_{L^2(\cD)}.
\end{align*}
Taking expectations and using Lemma~\ref{lem:Hm1} for $k=0$ and Lemma~\ref{lem:ApxHminHigherOrder} and \ref{lem:L2Estimate} for $k\in \mathbb{N}$ concludes the proof.
\end{proof}
Let us interpret the theorem. Our initial goal is the computation of the solution $u_h \in V_h$ to \eqref{eq:PmpVh}. However, the computation of $u_h$ is often not possible or very expensive due to quadrature of the right-hand side. Therefore, we replace $f$ by its piecewise polynomial approximation $\Pi_k f$. This causes an additional error that is bounded by the data-oscillation. This additional term is in most applications of higher order, consequently allowing for the same convergence rates, even in adaptive mesh refinement schemes, cf.~\cite{CasconKreuzerNochettoSiebert08,CarstensenFeischlPagePreatorius14}. Unfortunately, computing $\Pi_k f$ exactly is again often challenging or impossible. Therefore, we approximate $\Pi_k f$ by $\hat{\Pi} f$, leading to a second additional error. In expectation this error is also (almost) bounded by the data-oscillation -- thus it does not spoil the convergence rate. However, in the applications in mind we do not want to compute the expectation but rather a single sample. In this case, we can bound the likelihood of large resulting errors by the following lemma.
\begin{lemma}[High-probability bound]\label{lem:Markov}
With $\mathtt{ExpErr}^2 \coloneqq \mathbb E\bigl[\lVert f-\hat \Pi f\rVert_{H^{-1}(\cD)}^2\bigr]$ one has
\begin{align*}
\mathbb P \big(\alpha\, \mathtt{ExpErr}^2 \leq \lVert f-\hat \Pi f\rVert_{H^{-1}(\cD)}^2 \big)
\leq \frac{1}{\alpha}\qquad\text{for all }\alpha > 0.
\end{align*}
\end{lemma}
\begin{proof}
Since $0 \leq \lVert f-\hat \Pi f\rVert_{H^{-1}(\cD)}^2$, one has for any $\alpha > 0$ the bound
\begin{align*}
\alpha \, \mathtt{ExpErr}^2\, \mathbb P \big(\alpha\mathtt{ExpErr}^2 \leq \lVert f-\hat \Pi f\rVert_{H^{-1}(\cD)}^2 \big) \leq \mathtt{ExpErr}^2.
\end{align*}
Rearranging the terms concludes the proof.
\end{proof}
Let $\hat u_h(\omega) \in V_h$ be our computed approximation for some specific sample $\hat \Pi(\omega) f$ with $\omega \in \Omega$ as right-hand side. The lemma states in combination with Theorem~\ref{thm:Main1} that the additional squared error caused by the randomized load approximation is bounded with probability larger than $1-1/\alpha$ by
\begin{align*}
\lVert \nabla \bar u_h - \nabla \hat u_h(\omega) \rVert_{L^2(\cD)}^2 &< \alpha\, \mathbb{E}\big[ \lVert \Pi_k f - \hat \Pi f  \rVert_{H^{-1}(\cD)}^2\big],
\end{align*}
where the latter term is (almost) bounded by the data-oscillation, see Theorem \ref{thm:Main1}.
\begin{remark}[Right-hand sides in divergence form]
Our strategy and the resulting error estimates easily extend to right-hand sides in divergence form; that is, when we aim for an approximation of $u\in H^1_0(\cD)$ which solves with given $F\in L^2(\cD;\mathbb{R}^d)$ 
\begin{align*}
\int_\cD \nabla u \cdot \nabla v \dx = \int_\cD F \cdot \nabla v\dx\qquad\text{for all }v\in H^1_0(\cD).
\end{align*}
\end{remark}
\section{Numerical experiments}\label{sec:Exp}
We conclude our study with two numerical experiments. Throughout the experiments the domain $\cD = (0,1)^2$ is the unit square. We solve the discretized Poisson model problem in \eqref{eq:PmpVh} with the Lagrange finite element method; that is,
\begin{align*}
V_h \coloneqq S_0^k(\tria) \coloneqq \mathbb{P}_k(\tria) \cap H^1_0(\cD)\qquad\text{for } k \in \mathbb{N}.
\end{align*}
Our implementation uses NGSolve \cite{Schoeberl97,Schoeberl14} and can be found in \cite{StornCode}.
\subsection{Oscillating right-hand side}\label{subsec:Oscillation}
In this experiment we solve the Poisson model problem with oscillating right-hand side
\begin{align}\label{eq:rhs}
f(x,y) \coloneqq |\sin(\pi 2^L 3 x)|\qquad\text{and }L \coloneqq 5\qquad\text{for all }(x,y)\in \cD.
\end{align}
We use the lowest-order Lagrange space $V_h \coloneqq S^1_0(\tria)$ and evaluate the load $\int_{\cD} f v_h \dx$ for all $v_h \in V_h$ as follows:
\begin{enumerate}
\item We use the NGSolve routine $\lstinline{f*v*dx(bonus_intorder=r)}$ to evaluate the load numerically, which is with $r=0$ exact for polynomials up to the default integration order of NGSolve, leading to the results displayed on the left-hand side of Figure~\ref{fig:Exp1}. By using the parameter $r = 10$ we increase the order of polynomial exactness by 10, leading to the results on the right-hand side of Figure~\ref{fig:Exp1}.\label{itm:1}
\item We approximate $f$ by $\Pi_0 f$ and evaluate the load $\int_\cD \Pi_0 f\, v_h \dx$ exactly.  The piecewise constant function $\Pi_0 f$ is approximated by midpoint quadrature (Figure~\ref{fig:Exp1}, left) and by the numerically computed $L^2$ projection with right-hand side computed as in \ref{itm:1} with $r=10$ (Figure~\ref{fig:Exp1}, right).
\item We use the randomized projection operator $\hat \Pi_0 f$ and evaluate the load $\int_\cD \hat \Pi_0 f\, v_h \dx$ exactly. We use $N = 1$ samples per element (Figure~\ref{fig:Exp1}, left) and $N=20$ samples per element (Figure~\ref{fig:Exp1}, right)
\end{enumerate}
The underlying triangulations $\tria$ were obtained by uniform mesh refinements, leading to triangulations $\tria_0 \leq \dots \leq \tria_7$. In order to compute the error, we approximate the exact solution $u\in H^1_0(\cD)$ by the finite element solution $u \approx u_\textup{ref} \in S_0^1(\tria_{+})$, where $\tria_{+}$ results from two further uniform refinements of the finest mesh $\tria_7$ and a right-hand side approximated by $\hat{\Pi}_0 f$ with $N=100$ samples per element.
The resulting convergence history plot in Figure~\ref{fig:Exp1} illustrates that despite the highly oscillatory right-hand side the randomized method converges with the expected rate of $\mathcal{O}(\textup{ndof}^{-1/2})$ with $\textup{ndof} \coloneqq \dim V_h$ without any pre-asymptotic regime -- even with just $N=1$ samples per element.
Both lower order deterministic methods experience a pre-asymptotic regime without convergence -- a regime that gets larger when $L$ in \eqref{eq:rhs} is increased. The lack of convergence can be remedied by increasing the order of the deterministic quadrature schemes; however, it remains a regime where in particular the $L^2$ error of the randomized scheme remains smaller. 
\begin{figure}[ht!]
\begin{tikzpicture}
\begin{axis}[
clip=false,
width=.5\textwidth,
height=.45\textwidth,
ymode = log,
xmode = log,
xlabel = {$\textup{ndof}$},
cycle multi list={\nextlist MyColors},
scale = {1},
clip = true,
legend cell align=left,
legend style={legend columns=1,legend pos= south west,font=\fontsize{7}{5}\selectfont}
]
	\addplot table [x=ndof,y=H1err] {Experiments/Exp1_LO.txt};
	\addplot table [x=ndof,y=H1err_det] {Experiments/Exp1_LO.txt};	
	\addplot table [x=ndof,y=H1err_mc] {Experiments/Exp1_LO.txt};
	\addplot table [x=ndof,y=L2err] {Experiments/Exp1_LO.txt};
	\addplot table [x=ndof,y=L2err_det] {Experiments/Exp1_LO.txt};
	\addplot table [x=ndof,y=L2err_mc] {Experiments/Exp1_LO.txt};
	\legend{{$f$},{$\Pi_0 f$},{$\hat{\Pi}_0 f$}};
	\addplot[dash dot,sharp plot,update limits=false] coordinates {(1e1,1e0) (1e5,1e-2)};
\end{axis}
\end{tikzpicture}
\begin{tikzpicture}
\begin{axis}[
clip=false,
width=.5\textwidth,
height=.45\textwidth,
ymode = log,
xmode = log,
xlabel = {$\textup{ndof}$},
cycle multi list={\nextlist MyColors},
scale = {1},
clip = true,
legend cell align=left,
legend style={legend columns=1,legend pos= south west,font=\fontsize{7}{5}\selectfont}
]
	\addplot table [x=ndof,y=H1err] {Experiments/Exp1_HO.txt};
	\addplot table [x=ndof,y=H1err_det] {Experiments/Exp1_HO.txt};
	\addplot table [x=ndof,y=H1err_mc] {Experiments/Exp1_HO.txt};
	\addplot table [x=ndof,y=L2err] {Experiments/Exp1_HO.txt};
	\addplot table [x=ndof,y=L2err_det] {Experiments/Exp1_HO.txt};
	\addplot table [x=ndof,y=L2err_mc] {Experiments/Exp1_HO.txt};
	\legend{{$f$},{$\Pi_0 f$},{$\hat{\Pi}_0 f$}};
	\addplot[dash dot,sharp plot,update limits=false] coordinates {(1e1,1e0) (1e5,1e-2)};
\end{axis}
\end{tikzpicture}
\caption{Convergence history plot of the relative errors with respect to $\lVert \nabla \bigcdot \rVert_{L^2(\cD)}$ (solid line) and $\lVert \bigcdot \rVert_{L^2(\cD)}$ (dotted line) for various approximations of the right-hand side in the experiment of Section~\ref{subsec:Oscillation}. The dash-dotted line illustrates the slope $\textup{ndof}^{-1/2}$. On the left we evaluated $f$ using the standard NGSolve settings, $\Pi_0 f$ using midpoint quadrature, and $\hat{\Pi}_0 f$ using $N=1$ samples per element. On the right, we increased the standard NGSolve polynomial order of accuracy by 10, approximated $\Pi_0 f$ by a routine at least exact for polynomials of degree 10, and used $N= 20$ samples per element for the evaluation of $\hat{\Pi}_0 f$.  }\label{fig:Exp1}
\end{figure}
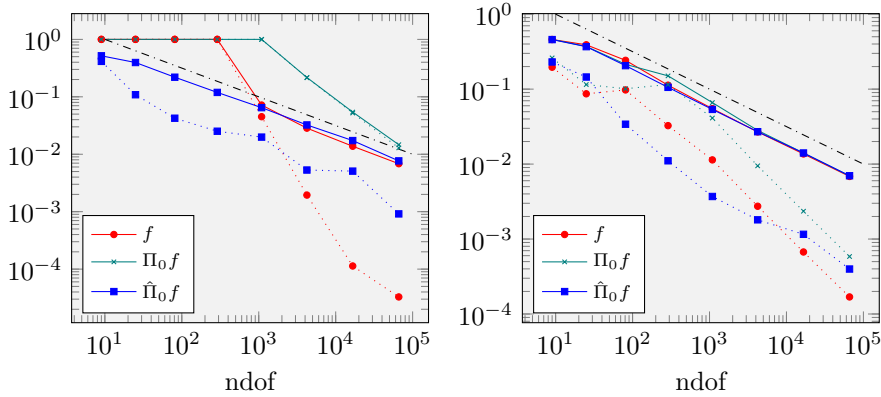
\subsection{Waterfall benchmark}\label{subsec:Waterfall}
In our second experiment we solve the benchmark problem from \cite[Sec.~4.2]{CarstensenGallistlHellwigWeggler14} with exact solution 
\begin{align*}
u(x,y) = x(x-1)y(y-1) \exp(-100(x-1/2)^2 - (y-117)^2 10^{-4}).
\end{align*}
We use quadratic finite elements  $V_h = S_0^2(\tria)$ and the standard adaptive mesh refinement routine with D\"orfler marking using the bulk parameter $\theta$ and residual-type error estimator as described in~\cite{Stevenson07}. We approximate the load as follows:
\begin{enumerate}
\item The right-hand sides $f$, $\Pi_0 f$, and $\hat{\Pi}_0 f$ are computed as in Section~\ref{subsec:Oscillation} (with the higher order parameters).
\item The approximation $\hat{\tilde{\Pi}}_1 f$ is computed with $M = 25$ samples per element for the calculation of $\tilde{\Pi}_1$ and $N = 10$ samples per element for the computation of the corrector $\hat{\Pi}_0$. 
\end{enumerate}
Figure~\ref{fig:Exp2} illustrates the convergence history of the resulting numerical approximations with respect to the $H^1(\cD)$ semi-norm (left) and the $L^2(\cD)$ norm (right). In both cases the results for the deterministic quadrature $f$ and the randomized approximation $\hat{\tilde{\Pi}}_1 f$ behave similarly.
The piecewise constant approximations $\Pi_0 f$ and $\hat{\Pi}_0 f$ result in slightly worse approximations, illustrating the benefits of higher-order randomized smoothers. This effect can be asymptotically emphasized even more when one considers finite element spaces of higher polynomial degree, however, the expected optimal rate of convergence is not attained in the observed pre-asymptotic regime. Figure~\ref{fig:Exp2} contains convergence history plots of three realizations of $\hat{\tilde{\Pi}}_1 f$. They differ slightly on coarse meshes, but the impact of the randomization seems to decrease as the mesh is refined. 

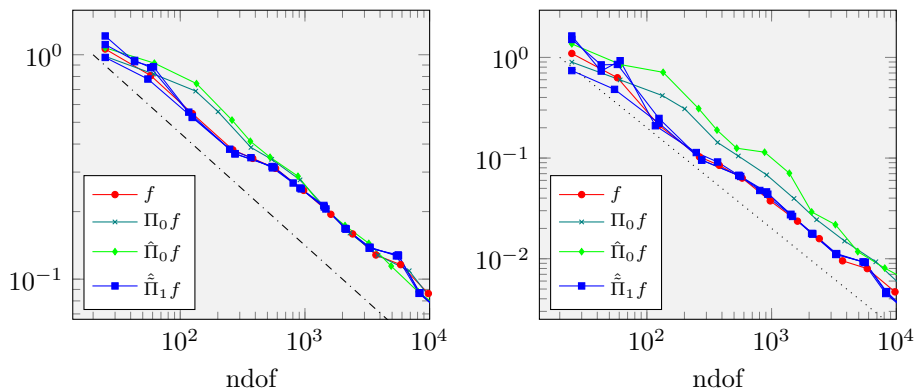
\begin{figure}[ht!]
\begin{tikzpicture}
\begin{axis}[
clip=false,
width=.5\textwidth,
height=.45\textwidth,
ymode = log,
xmode = log,
xlabel = {$\textup{ndof}$},
xmax = 10000,
cycle multi list={\nextlist MyColors2},
scale = {1},
clip = true,
legend cell align=left,
legend style={legend columns=1,legend pos= south west,font=\fontsize{7}{5}\selectfont}
]

	\addplot table [x=ndof,y=H1error] {Experiments/Exp2_HO_det.txt};
	\addplot table [x=ndof,y=H1error] {Experiments/Exp2_HO_Pi0_det.txt};	
	\addplot table [x=ndof,y=H1error] {Experiments/Exp2_LO_MC_1.txt};
	
	\addplot table [x=ndof,y=H1error] {Experiments/Exp2_HO_MC_1.txt};
	\addplot table [x=ndof,y=H1error] {Experiments/Exp2_HO_MC_2.txt};
	\addplot table [x=ndof,y=H1error] {Experiments/Exp2_HO_MC_3.txt};

	\addplot[dash dot,sharp plot,update limits=false] coordinates {(2e1,1e0) (2e5,1e-2)};
	\legend{{$f$},{$\Pi_0f$},{$\hat{\Pi}_0f$},{$\hat{\tilde{\Pi}}_1 f$}};
\end{axis}
\end{tikzpicture}
\begin{tikzpicture}
\begin{axis}[
clip=false,
width=.5\textwidth,
height=.45\textwidth,
ymode = log,
xmode = log,
xlabel = {$\textup{ndof}$},
xmax = 10000,
cycle multi list={\nextlist MyColors2},
scale = {1},
clip = true,
legend cell align=left,
legend style={legend columns=1,legend pos= south west,font=\fontsize{7}{5}\selectfont}
]
	\addplot table [x=ndof,y=L2error] {Experiments/Exp2_HO_det.txt};
	\addplot table [x=ndof,y=L2error] {Experiments/Exp2_HO_Pi0_det.txt};
		
	\addplot table [x=ndof,y=L2error] {Experiments/Exp2_LO_MC_1.txt};
	
	\addplot table [x=ndof,y=L2error] {Experiments/Exp2_HO_MC_1.txt};
	\addplot table [x=ndof,y=L2error] {Experiments/Exp2_HO_MC_2.txt};
	\addplot table [x=ndof,y=L2error] {Experiments/Exp2_HO_MC_3.txt};

	\addplot[dotted,sharp plot,update limits=false] coordinates {(2e1,1e0) (2e5,1e-4)};
	
	\legend{{$f$},{$\Pi_0f$},{$\hat{\Pi}_0f$},{$\hat{\tilde{\Pi}}_1 f$}};
\end{axis}
\end{tikzpicture}

\caption{Convergence history of the relative errors with respect to $\lVert \nabla \bigcdot \rVert_{L^2(\cD)}$ (left) and $\lVert \bigcdot \rVert_{L^2(\cD)}$ (right) for various approximations of the right-hand side in the experiment of Section~\ref{subsec:Waterfall}. We display three different realizations of $\hat{\tilde{\Pi}}_1 f$.
The dash-dotted line illustrates the rate $\mathcal{O}(\textup{ndof}^{-1/2})$, the dotted line illustrates $\mathcal{O}(\textup{ndof}^{-1})$.} \label{fig:Exp2}
\end{figure}%

\appendix
\section{Bounds for $\lVert \hat G_K^{-1}\rVert$}
In this appendix we verify the auxiliary result needed in the proof of Lemma~\ref{lem:L2Estimate}. For this, we fix $K\in \tria$ and recall the basis $(\psi_{1},\dots,\psi_{m})$ of $\mathbb{P}_k(K)$ being orthonormal in the sense of \eqref{eq:orthoBasis}. These basis functions define with $N\geq m$ uniformly and independently distributed random variables $X_1,\dots,X_{N} \sim \mathcal{U}(K)$ the empirical Gram matrix
\begin{align*}
G \coloneqq \frac{1}{N} \sum_{i=1}^{N} \psi(X_i) \psi(X_i)^\top.
\end{align*} 
Let $\lambda_\textup{min}(G) \geq 0$ denote its smallest eigenvalue and recall the constant $\Lambda_k < \infty$ defined in Lemma~\ref{lem:ChristoffelTypeQuant}. Moreover, recall the notation 
\begin{align*}
\psi(x) \coloneqq \big(\psi_1(x),\dots,\psi_m(x)\big)^\top \in \mathbb{R}^m\qquad\text{for all }x\in K.
\end{align*}
\begin{lemma}[Tail estimate for the inverse empirical Gram matrix]\label{lem:tail_Ginv}
Let $k > 0$.
There exist constants $\varepsilon_0 = \varepsilon_0(d,k) \in(0,1)$ and $C=C(d,k)<\infty$ such that
\begin{align*}
\mathbb P\big(\lambda_{\min}(G)\leq \varepsilon\big) \leq 12^m\Lambda_k^m C^N \varepsilon^{\frac{N}{4k}-m}\qquad\text{for all }\varepsilon\in(0,\varepsilon_0].
\end{align*}
Equivalently, we have
\begin{align*}
\mathbb P\big(t \leq \lVert G^{-1}\rVert\big) \leq 12^m \Lambda_k^m C^N t^{-\left(\frac{N}{4k}-m\right)}\qquad\text{for all }t\geq 1/\varepsilon_0.
\end{align*}
\end{lemma}

\begin{proof}
\textit{Step 1 (Small-ball estimate).}
Let us fix some $u\in S^{m-1}\coloneqq \lbrace v\in \mathbb{R}^m \colon \lVert v \rVert_{\ell^2} = 1\rbrace$ and define the polynomial $p_u(x)\coloneqq u\cdot  \psi(x) \in\mathbb P_k(K)$. The orthonormality of the basis yields $\mathbb E[p_u(X)^2]= \lVert u \rVert_{\ell^2}^2 = 1$.
By the Carbery--Wright inequality \cite[Thm.~8]{CarberyWright01}, there exists a uniformly bounded constant $C_0=C_0(d,k) < \infty$ such that
\begin{align}\label{eq:CW_used}
\mathbb P\big(|p_u(X)|\leq \tau\big)\le C_0\,\tau^{1/k} \qquad\text{for all }\tau > 0.
\end{align}
We fix $\varepsilon > 0$, set $\tau\coloneqq\sqrt{2\varepsilon}$, and abbreviate
$\mathbb{P}_{\varepsilon,u} \coloneqq \mathbb P(|p_u(X)|\leq \sqrt{2\varepsilon})\le C_0(2\varepsilon)^{1/(2k)}$.
We have the identities
\begin{align*}
u^\top G u = \frac{1}{N}\sum_{i=1}^{N} \big(u\cdot  \psi(X_i)\big)^2 = \frac{1}{N}\sum_{i=1}^{N}p_u(X_i)^2 .
\end{align*}
Consequently, if $u^\top G u \leq \varepsilon$, at least $\lceil N/2 \rceil$ indices must satisfy
$|p_u(X_i)|\leq \sqrt{2\varepsilon}$.
Therefore, we obtain with the bound $\binom Nj \leq 2^N $ implied by the identity $\sum_{j=0}^N \binom Nj = 2^N$ as well as $2^N \leq 4^j$ for all $j\geq \lceil N/2 \rceil$ the estimate
\begin{align*}
\mathbb P\big(u^\top G u\leq \varepsilon\big) &\leq \mathbb P\Big(\lceil N/2 \rceil \leq \#\{i = 1,\dots, N\colon  |p_u(X_i)|\leq \sqrt{2\varepsilon}\}\Big)\\
&\leq \sum_{j=\lceil N/2\rceil}^{N}\binom Nj\,\mathbb{P}_{\varepsilon,u}^j \leq \sum_{j=\lceil N/2\rceil}^{N}(4\mathbb{P}_{\varepsilon,u})^j .
\end{align*}
Let $\varepsilon_0\in (0,1)$ be sufficiently small such that $\mathbb{P}_{\varepsilon,u} \leq C_0(2\varepsilon)^{1/(2k)} \leq 1/8$ for all $\varepsilon\leq\varepsilon_0$. Then any $\varepsilon\in(0,\varepsilon_0]$ satisfies
with the uniformly bounded constant $C_1 = C_1(d,k) \coloneqq (2^{1/(2k)}8C_0)^{1/2} < \infty$ the estimate
\begin{align}\label{eq:one_dir}
\begin{aligned}
\mathbb P\big(u^\top G u\leq \varepsilon\big) &\leq \sum_{j= \lceil N/2 \rceil}^N (4\mathbb{P}_{\varepsilon,u})^j \leq \frac{(4\mathbb{P}_{\varepsilon,u})^{\lceil N/2\rceil}}{1 - 4\mathbb{P}_{\varepsilon,u}} \leq  (8\mathbb{P}_{\varepsilon,u})^{N/2} \\
& \leq  \big(8C_0 (2\varepsilon)^{1/(2k)} \big)^{N/2}  = C_1^{N}\,\varepsilon^{N/(4k)}.
\end{aligned}
\end{align}

\textit{Step 2 (Net argument).}
Define for any $u\in S^{m-1}$ the function $F(u)\coloneqq u^\top G u$. Moreover, set $v_i \coloneqq \psi(X_i) \in L^2(\Omega;\mathbb R^m)$ for $i=1,\dots,N$. We have $\mathbb E[v_iv_i^\top]=\identity_m$  by orthonormality, $\lVert v_i \rVert_{\ell^2}^2\leq \Lambda_k$ by Lemma~\ref{lem:ChristoffelTypeQuant}, and by definition
\begin{align*}
G = \frac{1}{N} \sum_{i=1}^N v_i v_i^\top \qquad\text{and thus }\qquad F(u) = \frac{1}{N} \sum_{i=1}^N (v_i \cdot u)^2 .
\end{align*}
We obtain for each $u,u'\in S^{m-1}$ the estimate
\begin{align*}
\big|(v_i\cdot u)^2- (v_i\cdot u')^2\big| &=|\big( v_i\cdot (u+u')\big)\big( v_i\cdot (u-u')\big)| \leq \lVert v_i\rVert_{\ell^2}^2\,\lVert u+u'\rVert_{\ell^2}\,\lVert u-u'\rVert_{\ell^2} \\
&\leq 2\Lambda_k\,\lVert u-u'\rVert_{\ell^2}.
\end{align*}
This shows the Lipschitz continuity
\begin{align*}
|F(u)-F(u')|\leq 2\Lambda_k\lVert u-u'\rVert_{\ell^2}\qquad\text{for all } u,u'\in S^{m-1}.
\end{align*}
Let $\delta\coloneqq \varepsilon/(4\Lambda_k)$ with $\varepsilon < \min\lbrace 2\varepsilon_0/3, 4 \Lambda_k\rbrace$ and let $\mathcal N_\delta\subset S^{m-1}$ be a $\delta$-net.
If $\lambda_{\min}(G)=\inf_{u\in S^{m-1}}F(u)\leq \varepsilon$, choose $u\in S^{m-1}$ with $F(u)\le \varepsilon$ with neighbor $u'\in\mathcal N_\delta$ satisfying $\lVert u-u'\rVert_{\ell^2}\leq\delta$.
By Lipschitz continuity we obtain the bound
\begin{align*}
F(u')\leq F(u)+2\Lambda_k\delta\leq \varepsilon+\frac{\varepsilon}{2}=\frac32\varepsilon.
\end{align*}
This observation and \eqref{eq:one_dir} yield
\begin{align*}
\mathbb P\big(\lambda_{\min}(G)\leq \varepsilon\big) \leq \sum_{u'\in\mathcal N_\delta}\mathbb P\Big(F(u')\leq \frac32\varepsilon\Big) \leq |\mathcal N_\delta|\, C^N_1\,\Big(\frac32\Big)^{N/(4k)} \,\varepsilon^{N/(4k)}.
\end{align*}
Since $\delta=\varepsilon/(4\Lambda_k)\leq 1$, the net cardinality is bounded by \cite[Cor.~4.2.11]{Vershynin18}
\begin{align*}
|\mathcal N_\delta|\le \Big(\frac{3}{\delta}\Big)^m=\Big(\frac{12\,\Lambda_k}{\varepsilon}\Big)^m.
\end{align*}
Combining these statements leads with $C_2 \coloneqq C_1(3/2)^{1/(4k)}$ to the bound 
\begin{align}\label{eq:tail_lmin_clean}
\mathbb P\big(\lambda_{\min}(G)\le \varepsilon\big) \leq (12\Lambda_k)^mC_2^N \varepsilon^{N/(4k)-m}.
\end{align}
This yields the first statement of the lemma. Since $\lVert G^{-1}\rVert=1/\lambda_{\min}(G)$, the second statement follows with $\varepsilon\coloneqq 1/t$.
\end{proof}
The bound for the tail has a factor $C^N$ that grows fast in $N$. This growth can be compensated by the even faster decay in $t$, as displayed in the following corollary.
\begin{corollary}[$N$-robust bound]\label{cor:NrobustBound}
Let $N \geq N_0 \geq m$ and let $k > 0$. Then there exist constants $c = c(d,k) < \infty$ and $C_\star = C_\star(d,k,N_0) < \infty$ such that 
\begin{align*}
\mathbb P\big(t \leq \lVert G^{-1}\rVert\big) \leq C_\star t^{-\left(\frac{N_0}{4k}-m\right)}\qquad\text{for all }t\geq c.
\end{align*}
\end{corollary}
\begin{proof}
Let $C = C(d,k) < \infty$ and $\varepsilon_0\in (0,1)$ denote the constants in Lemma~\ref{lem:tail_Ginv}.
Then we have for all $t \geq \max\lbrace 1/\varepsilon_0,C^{4k}\rbrace$ that
\begin{align*}
\frac{C}{t^{1/(4k)}} \leq 1,\qquad\text{implying}\qquad
\Big( \frac{C}{t^{1/(4k)}} \Big)^{N} \leq \Big( \frac{C}{t^{1/(4k)}} \Big)^{N_0}.
\end{align*}
Combining this observation with Lemma~\ref{lem:tail_Ginv} verifies for all  $t \geq \max\lbrace 1/\varepsilon_0,C^{4k}\rbrace$  that
\begin{align*}
&\mathbb P\big(t \leq \lVert G^{-1}\rVert\big) \leq 12^m \Lambda_k^m C^N t^{-\left(\frac{N}{4k}-m\right)} = 12^m \Lambda_k^m t^m \Big(\frac{C}{t^{1/(4k)}}\Big)^N \\
& \qquad\leq 12^m \Lambda_k^m t^m \Big(\frac{C}{t^{1/(4k)}}\Big)^{N_0} = 12^m \Lambda_k^m C^{N_0} t^{-\left(\frac{N_0}{4k}-m\right)}.
\end{align*}
This concludes the proof.
\end{proof}

\subsection*{Acknowledgment}
The author would like to thank Mario Ullrich for helpful discussions and for drawing his attention to several relevant references on randomized approximation and least-squares methods.
\printbibliography

\end{document}